\documentclass[12pt,reqno]{amsart}
\usepackage{amssymb,amsmath,amsthm,amsfonts,latexsym}
\usepackage{mathrsfs}

\usepackage{geometry}
\newtheorem{theorem}{Theorem}[section]
\newtheorem{lemma}[theorem]{Lemma}

\newtheorem{proposition}[theorem]{Proposition}

\numberwithin{equation}{section}

\def \A {{\mathbb A}}

\def \C {{\mathbb C}}

\def \N {{\mathbb N}}
\def \P {{\mathbb P}}

\def \R {{\mathbb R}}
\def \U {{\mathbb U}}

\def \Z {{\mathbb Z}}
\def \le{{\,\leqslant\,}}
\def \ge{{\,\geqslant\,}}

\def\pmod #1{({\rm mod}\ #1)}

\def\rank {{\rm{rank}}}
\def \codim {{\rm{codim}}}
\begin{document}
\title[On forms in prime variables]{On forms in prime variables}
\author{Jianya Liu}
\address{School of Mathematics and Data Science Institute\\ Shandong University \\ Jinan  250100 \\ China}
\email{jyliu@sdu.edu.cn}
\author{Lilu Zhao}
\address{School of Mathematics \\ Shandong University \\ Jinan  250100 \\ China}
\email{zhaolilu@sdu.edu.cn}

\begin{abstract} Let $F_1,\ldots,F_R$ be homogeneous polynomials of degree $d\ge 2$ with integer coefficients in $n$ variables,
and let $\mathbf{F}=(F_1,\ldots,F_R)$. Suppose that $F_1,\ldots,F_R$ is a non-singular system and $n\ge 4^{d+2}d^2R^5$. We prove that there are infinitely many solutions to $\mathbf{F}(\mathbf{x})=\mathbf{0}$ in prime coordinates if
(i) $\mathbf{F}(\mathbf{x})=\mathbf{0}$ has a non-singular solution
over the $p$-adic units $\U_p$ for all prime numbers $p$, and
(ii) $\mathbf{F}(\mathbf{x})=\mathbf{0}$ has a non-singular solution in the open cube $(0,1)^n$.
\end{abstract}

\maketitle


{\let\thefootnote\relax}

{\let\thefootnote\relax}

{\let\thefootnote\relax\footnotetext{This work is supported by the NSFC grants 12031008 and 11922113.}}

\section{Introduction}

It is a fundamental ambition in number theory to establish the solubility of diophantine equations in many variables. With the exception of diagonal forms and their close kin, the problem of solving a general homogeneous equation is of considerable complexity. The pioneering work of Birch \cite{Birch} provided a method to find integral solutions in such a problem whenever the number of variables is sufficiently large. Birch's work has extensive and significant impact on number theory, and in particular his method has been developed to deal with numerous problems. Skinner \cite{Skinner} (see also \cite{Skinner1997}) considered the generalisation to arbitrary number fields.
Recently, Browning and Heath-Brown \cite{BHb} succeeded in solving forms in many variables and differing degrees.

In this paper, we are interested in finding prime solutions. There is a long history of investigation of prime solutions to linear diophantine equations. Refining the earlier conditional result of Hardy and Littlewood, Vinogradov successfully solved the ternary Goldbach problem
 by showing every sufficiently large odd number can be represented as a sum of three primes. Green and Tao \cite{GT}  made a breakthrough on the existence of arbitrarily long arithmetic progressions in prime numbers. Recently, Zhang \cite{Zhang} and Maynard \cite{Maynard} made significant contributions to the distribution of primes with bounded gaps.  For applications of the classical Vinogradov method to nonlinear diagonal equations, one may refer to \cite{KawadaWooley,Zhao} for the development in the Waring-Goldbach problem.  We shall consider prime solutions to a system of general nonlinear equations of the same degree. For $1\le i\le R$, let
$$F_i(\mathbf{x})=F_i(x_1,\ldots,x_n)\in \Z[x_1,\ldots,x_n]$$
 be homogeneous polynomials of degree $d\ge 2$. Throughout this paper, we use bold face letters to denote vectors whose dimensions are clear from the context.

 Let $\mathbf{F}=(F_1,\ldots,F_R)$. In \cite{BDLW}, Br\"udern et al. studied the equations \begin{align}\label{equationF=0} \mathbf{F}(x_1,\ldots,x_n)=\mathbf{0}\end{align}
 with solutions $(x_1,\ldots,x_n)=(c_1p_1,\ldots,c_np_n)\in \Z^{n}$, where $p_1,\ldots,p_n$ are pairwise distinct primes and $c_1,\ldots,c_n$ are some fixed integers.  In the very recent impressive work of Cook and Magyar \cite{CM}, they were able to handle \eqref{equationF=0} in prime variables by assuming that $n$ is extremely large.

Let $V\subseteq \A^n$ be the variety
 \begin{align*}V=V_{\mathbf{F}}=\big\{\mathbf{x}\in \A^n:\ \mathbf{F}(\mathbf{x})=\mathbf{0}\big\}.\end{align*}
 Let $\P$ denote the set of prime numbers. We study $V(\P)$, which is the set of points in $V$ with prime coordinates, that is
 \begin{align*}V(\P)=\big\{\mathbf{x}\in \P^n:\ \mathbf{F}(\mathbf{x})=\mathbf{0}\big\}.\end{align*}
 We introduce the Jacobian matrix
\begin{align*}
J_{\mathbf{F}}(\mathbf{x})=\Big(\frac{\partial F_i}{\partial x_j}(\mathbf{x})\Big)_{\substack{1\le i\le R\\ 1\le j\le n}}.\end{align*}
Following Browning and Heath-Brown \cite{BHb}, we say $F_1,\ldots,F_R$ is a non-singular system if $\rank(J_{\mathbf{F}}(\mathbf{x}))=R$ for every nonzero $\mathbf{x}\in V_\mathbf{F}$.

Our main result is as follows.
\begin{theorem}\label{theorem1}Let $F_1,\ldots,F_R\in \Z[x_1,\ldots,x_n]$ be homogeneous polynomials of degree $d\ge 2$.
Suppose that $F_1,\ldots,F_R$ is a non-singular system and that
$$n\ge 4^{d+2}d^2R^5.$$
Then $V(\P)$ is Zariski dense in $V$ provided that the following two conditions hold:

(i) $\mathbf{F}(\mathbf{x})=\mathbf{0}$ has a non-singular solution over the $p$-adic units $\U_p$ for all prime numbers $p$, and

(ii) $\mathbf{F}(\mathbf{x})=\mathbf{0}$ has a non-singular solution in the open cube $(0,1)^n$.
\end{theorem}

Theorem \ref{theorem1} discusses the solutions to \eqref{equationF=0} from the point view
of Bourgain, Gamburd and Sarnak \cite{BGS}. In particular, if the above local conditions in Theorem \ref{theorem1} are satisfied, then there are infinitely many solutions to \eqref{equationF=0} in prime variables.
Theorem \ref{theorem1} is new even in the case $R=1$.
Cook and Magyar \cite{CM} proved that for each $R$ and $d$, there exists a number $\chi(R,d)$ such that if $n-\dim V_{\mathbf{F}}^\ast\ge \chi(R,d)$ then the above smooth Hasse principle holds, where $V_{\mathbf{F}}^\ast$ is the singular locus (in the sense of Birch)
\begin{align*}V_{\mathbf{F}}^\ast=\{\mathbf{x}\in \A^n:\ \rank J_{\mathbf{F}}(\mathbf{x})<R\}.\end{align*}
For comparison, we mention some remarks in \cite{CM}. Cook and Magyar \cite{CM} pointed out that the quantitative aspects of $\chi(R,d)$ in their proof are in general extremely poor, and in particular, the numbers $\chi(1,d)$ already exhibit tower type behavior in the degree $d$. They also pointed out that for a system of quadratic forms, $\chi(R,2)\le 2^{2^{cR^2}}$ with some absolute constant $c>0$.

The improvement on the number of variables in Birch's type theorem has been stimulating numerous works. The classical works of Heath-Brown \cite{HB1983,HB2007} give impressive improvement on Birch's theorem when $d=3$ and $R=1$. Browning and Prendiville \cite{BP} obtained improvement for all degree $d\ge 4$ in the case $R=1$. For a cubic form over number fields, Skinner's result was improved by Browning and Vishe \cite{BV2014}.  Very recently, Rydin Myerson \cite{Myerson} obtained considerable improvement for systems of quadratic forms.

The proof of Cook and Magyar \cite{CM} is based on the regularization process developed by Schmidt \cite{Schmidt}. Our method in the proof of Theorem \ref{theorem1} is quite different. We shall develop a new approach to Theorem 1.1 based on the difference argument for exponential summations over special sequences of natural numbers, such as primes. As is well-known, there is a considerable difficulty to apply the difference argument to multiple exponential summations, when the variables are restricted to special sequences. We overcome this difficulty by investigating the mean value result.

We establish our theorems by applying the Hardy-Littlewood method. The general philosophy underlying applications of the Hardy-Littlewood method suggests one may expect to establish an asymptotic formula for the number of prime solutions to \eqref{equationF=0} in a finite box. Indeed, Theorem \ref{theorem1} was obtained from such a quantitative result. In order to state the asymptotic formula precisely, we introduce some notations.

Let $\mathfrak{B}$ be a fixed box in $n$-dimensional space determined by
\begin{align*}b_j'< x_j\le b_{j}'',\end{align*}
where $0<b_j'<b_{j}''<1$ are fixed constants for $1\le j\le n$. Suppose that $P$ is sufficiently large. We use $P\mathfrak{B}$ to denote the set of all vectors $\mathbf{x}$ with $P^{-1}\mathbf{x}\in \mathfrak{B}$. We define
\begin{align}\label{defineNFP}N_{\mathbf{F}}(P)=\sum_{\substack{\mathbf{x}\in P\mathfrak{B}
\\ \mathbf{F}(\mathbf{x})=\mathbf{0}}}\Lambda(x_1)\cdots \Lambda(x_n),\end{align}
where $\Lambda(\cdot)$ denotes the Von Mangoldt function.

To explain the asymptotic behaviour, we need to introduce the singular series $\mathfrak{S}_\mathbf{F}^\ast$ and the singular integral $\mathfrak{I}_\mathbf{F}$, which are standard in the application of the Hardy-Littlewood method. The precise definitions will be given in \eqref{definesingularseries} and \eqref{definesingularintegral}, respectively.

 Now we introduce our quantitative result for $N_\mathbf{F}(P)$.

\begin{theorem}\label{theorem2}Let $F_1,\ldots,F_R\in \Z[x_1,\ldots,x_n]$ be homogeneous polynomials of degree $d\ge 2$. Let $N_\mathbf{F}(P)$ be defined in \eqref{defineNFP}.
Suppose that $F_1,\ldots,F_R$ is a non-singular system and that
$$n\ge 4^{d+2}d^2R^5.$$
For any constant $A>0$, we have
\begin{align*}N_\mathbf{F}(P)=\mathfrak{S}_\mathbf{F}^\ast \mathfrak{I}_\mathbf{F} P^{n-Rd}+O\big(P^{n-Rd}(\log P)^{-A}\big),\end{align*}
where the singular series $\mathfrak{S}_\mathbf{F}^\ast$ and the singular integral $\mathfrak{I}_\mathbf{F}$ are given in \eqref{definesingularseries} and \eqref{definesingularintegral}, respectively.
\end{theorem}

Let
\begin{align*}N_{\mathbf{F}}^\ast(P)=\sum_{\substack{\mathbf{p}\in P\mathfrak{B}
\\ \mathbf{F}(\mathbf{p})=\mathbf{0}}}(\log p_1)\cdots (\log p_n),\end{align*}
where the multiple summation is taken over $\mathbf{p}=(p_1,\ldots,p_n)$ in the box $P\mathfrak{B}$, and $p_1,\ldots,p_n$ are prime numbers. Then $N_\mathbf{F}^\ast(P)$ counts the weighted number of prime solutions to \eqref{equationF=0} in a finite box. Theorem \ref{theorem2} holds as well with $N_\mathbf{F}(P)$ replaced by $N_\mathbf{F}^\ast(P)$.

If $\mathbf{F}(\mathbf{x})=\mathbf{0}$ has a non-singular solution over $\U_p$ for all prime numbers $p$, then $\mathfrak{S}_\mathbf{F}^\ast>0$. Moreover, if $\mathbf{F}(\mathbf{x})=\mathbf{0}$ has a non-singular solution in $(0,1)^n$, then we can choose a suitable box $\mathfrak{B}$ such that $\mathfrak{I}_\mathbf{F}>0$. In particular, subject to smooth local conditions, Theorem \ref{theorem2} yields
\begin{align}\label{lowerNFP2}N_{\mathbf{F}}^\ast(P)\gg P^{n-Rd}.\end{align}
 Now Theorem \ref{theorem1} follows from  \eqref{lowerNFP2} by the standard argument (see the proof of Corollary 2.3 in \cite{LS}). Therefore, our task is to prove Theorem \ref{theorem2}.

 The method in the proof of Theorem \ref{theorem2} is based on the difference argument, and thus it works well to deal with a system of general polynomials of the same degree subject to some assumptions on the highest degree homogeneous parts.

\vskip3mm

We introduce some notations in Section 2 and we shall study some classical lemmas in Section 3. We prepare some technical mean value estimates in Sections 4-5. We study exponential sums over primes in Section 6. Then in Section 7 we deduce an acceptable minor arcs estimate subject to some technical conditions on $\mathbf{F}$. Section 8 is devoted to geometric considerations. As an application, we remove the technical conditions introduced in Section 7. We explain the singular series and the singular integral in Section 9, and complete the proof of Theorem \ref{theorem2} in Section 10.

\vskip3mm

\section{Notations}

As usual, we write $e(z)$ for $e^{2\pi iz}$. Throughout this paper, we assume
that $P$ is sufficiently large. We use $\ll$ and
$\gg$ to denote Vinogradov's well-known notations. The implied constant may depend on the polynomials $F_1,\ldots,F_R$. Denote by $\phi(q)$
Euler's totient function. For $\alpha \in \R$, we use $\|\alpha\|$ to denote $\min_{m\in \Z}|\alpha-m|$.

The letter $\varepsilon$ denotes a sufficiently small positive real number. Any statement in which $\varepsilon$ occurs holds for each fixed $\varepsilon > 0$, and any implied constant in such a statement is allowed to depend on $\varepsilon$.

For $\mathbf{x}=(x_1,\ldots,x_n)\in \Z^n$, we use $\mathfrak{A}(\mathbf{x})$ to indicate that
$\mathfrak{A}(x_j)$ holds for all $1\le j\le n$. The meaning will
be clear from the context. For example, we use $1\le \mathbf{x}\le X$
to denote $1\le x_j\le X$ for all $1\le j\le
n$.

The letters $F,G,H$ and $f,g,h$ are reserved for polynomials. A homogeneous polynomial of degree $d$ with coefficients integral and symmetric will be called a form of degree $d$.

For any fixed natural number $i\in \N$, we assume that $\{\lambda_i(x)\}_{x=1}^\infty$  is a real sequence satisfying
\begin{align*}|\lambda_i(x)|\le \log x \ (x\ge 1).\end{align*}
Then for a vector $\mathbf{x}=(x_1,\ldots,x_n)\in \N^n$, we use $\lambda(\mathbf{x})$ to denote the product
\begin{align*}\lambda(\mathbf{x})=\lambda_1(x_1)\cdots \lambda_n(x_n).\end{align*}

In order to apply the circle method, we introduce the exponential sum
\begin{align}\label{definegeneratingS}S_{\mathbf{F}}(\boldsymbol{\alpha})=\sum_{1\le \mathbf{x}\le P}\lambda(\mathbf{x})
e\big(\boldsymbol{\alpha}\cdot \mathbf{F}(\mathbf{x})\big).\end{align}
 We define the major arcs
\begin{align}\label{defineMQ}\mathfrak{M}(Q)=\bigcup_{1\le q\le
Q}\bigcup_{\substack{1\le a_1,\ldots,a_R\le q
 \\ (a_1,\ldots,a_R,q)=1}}\mathfrak{M}(q,\mathbf{a};Q),\end{align}
where
\begin{align*}\mathfrak{M}(q,\mathbf{a};Q)=\Big\{(\alpha_1,\ldots,\alpha_R)\in \R^R:\ \big|\alpha_i-\frac{a_i}{q}\big|\le \frac{Q}{qP^d} \ \textrm{ for all }\ 1\le i\le R\Big\}.\end{align*}
Then for $Q\le \frac{1}{4}P^{\frac{Rd}{R+1}}$, we define the minor arcs
\begin{align}\label{definemQ}\mathfrak{m}(Q)=[P^{-d/2},1+P^{-d/2}]^{R}\setminus\mathfrak{M}(Q).\end{align}
Note that when $Q\le \frac{1}{4}P^{\frac{d}{2}}$, the unions in \eqref{defineMQ} are pairwise disjoint.

\vskip3mm
\section{Classical lemmas from Birch's work}

This section is devoted to the review of classical lemmas in Birch's work \cite{Birch}. Essentially, all results in this section can be established  by the arguments in Sections 2-3 of Birch \cite{Birch}. There are several minor differences.
For example, in the application of the Hardy-Littlewood method, Birch defines the major arcs $\mathfrak{M}(Q)$ by choosing $Q$ to be
$P^{\theta}$ with a small number $\theta=\theta_{d,R,n}>0$. However, to study diophantine equations in prime variables, we need to work with the major arcs $\mathfrak{M}(X)$ by choosing $X=(\log P)^{A}$ with a sufficiently large number $A=A_{d,R,n}>0$. Moreover, we have to replace some of $P^\varepsilon$ terms in Birch's lemmas by a power of $\log P$. Although these differences do not produce any difficulties, we plan to rewrite these lemmas and include some proofs for completeness.

Suppose that $\mathfrak{f}=(f_1,\ldots,f_R)$ and $f_1(\mathbf{x}),\ldots,f_R(\mathbf{x})$ are forms of degree $d$ in $n$ variables, that is,
\begin{align}\label{coeff1}
f_i(\mathbf{x})=\sum_{1\le j_0,j_1,\ldots,j_{d-1}\le n}f_{j_0,j_1,\ldots,j_{d-1}}^{(i)}x_{j_0}x_{j_1}\cdots x_{j_{d-1}},\end{align}
where the coefficients $f_{j_0,j_1,\ldots,j_{d-1}}^{(i)}$ are integers and symmetric.

For $1\le j\le n$ and $d-1$ vectors $\mathbf{x}^{(1)},\ldots,\mathbf{x}^{(d-1)}$ of dimension $n$, we define
\begin{align*}\Psi_j^{(i)}(\mathbf{x}^{(1)},\ldots,\mathbf{x}^{(d-1)};\mathfrak{f})=d!\sum_{1\le j_1,\ldots,j_{d-1}\le n}f_{j,j_1,\ldots,j_{d-1}}^{(i)}\mathbf{x}^{(1)}_{j_1}\cdots \mathbf{x}^{(d-1)}_{j_{d-1}}.\end{align*}
Then for $\boldsymbol{\alpha}=(\alpha_1,\ldots,\alpha_R)\in \R^R$, we introduce
\begin{align*}\Phi_j(\boldsymbol{\alpha};\mathbf{x}^{(1)},\ldots,\mathbf{x}^{(d-1)};\mathfrak{f})=\sum_{i=1}^R\alpha_i\Psi_j^{(i)}(\mathbf{x}^{(1)},\ldots,\mathbf{x}^{(d-1)};\mathfrak{f})
\end{align*}
and
\begin{align*}\Gamma_{\mathfrak{f}}(\boldsymbol{\alpha}; P)=\sum_{\substack{\mathbf{x}^{(1)},\ldots,\mathbf{x}^{(d-1)}
\\ |\mathbf{x}^{(1)}|\le P,\ldots,|\mathbf{x}^{(d-1)}|\le P}}\prod_{j=1}^{n}\min\big(P,\ \big\| \Phi_j(\boldsymbol{\alpha};\mathbf{x}^{(1)},\ldots,\mathbf{x}^{(d-1)};\mathfrak{f})\big\|^{-1}\big).\end{align*}
In order to introduce the upper bound for $\Gamma_{\mathfrak{f}}(\boldsymbol{\alpha}; P)$, we define
\begin{align}\label{defineNXY}N(X,Y^{-1};\boldsymbol{\alpha})\end{align}
to be the number of $(d-1)$-tuples of integer points $\mathbf{x}^{(1)},\ldots,\mathbf{x}^{(d-1)}$ which satisfy
\begin{align}\label{boundd-1tuples}|\mathbf{x}^{(1)}|\le X,\ldots,|\mathbf{x}^{(d-1)}|\le X,\end{align}
and
\begin{align*}\big\|\Phi_j(\boldsymbol{\alpha};\mathbf{x}^{(1)},\ldots,\mathbf{x}^{(d-1)};\mathfrak{f})\big\|<Y^{-1} \textrm{ for all } 1\le j\le n.\end{align*}

\begin{lemma}\label{lemmaboundgammatoN}We have
\begin{align*}\Gamma_{\mathfrak{f}}(\boldsymbol{\alpha}; P)\ll P^n(\log^n P)N(P,P^{-1};\boldsymbol{\alpha}) .\end{align*}\end{lemma}

Lemma \ref{lemmaboundgammatoN} can be established by using the argument in the proof of Lemma 3.2 of Davenport \cite{Davenport}.

As an application of Lemma 3.3 in \cite{Davenport} (see also Lemma 2.3 in \cite{Birch}), we have the following.
\begin{lemma}\label{lemmaNtoN}Suppose that $X\le P$. Then one has
\begin{align*}N(P,P^{-1};\boldsymbol{\alpha}) \ll P^{(d-1)n}X^{-(d-1)n} N(X,X^{d-1}P^{-d};\boldsymbol{\alpha}).\end{align*}\end{lemma}

Let $\Upsilon(\mathbf{x}^{(1)},\ldots,\mathbf{x}^{(d-1)};\mathfrak{f})$ denote the $R$ by $n$ matrix
\begin{align*}\Big(\Psi_j^{(i)}(\mathbf{x}^{(1)},\ldots,\mathbf{x}^{(d-1)};\mathfrak{f})\Big)_{\substack{1\le i\le R
\\ 1\le j\le n}}.\end{align*}
We use $\mathscr{L}$ to denote the affine locus in $(d-1)n$-dimensional space consisting of points $(\mathbf{x}^{(1)},\ldots,\mathbf{x}^{(d-1)})$ satisfying
\begin{align*}\rank\Upsilon(\mathbf{x}^{(1)},\ldots,\mathbf{x}^{(d-1)};\mathfrak{f})< R ,\end{align*}
and we  use $\mathscr{L}(X)$ to denote the number of integer points $(\mathbf{x}^{(1)},\ldots,\mathbf{x}^{(d-1)})$ in $\mathscr{L}$, which satisfy \eqref{boundd-1tuples}.

\begin{lemma}\label{lemmaoneofthree}Let $\kappa>0$ be a real number. Suppose that $X\le P$. With the notations above, at least one of the three following holds: either (i)
$$\Gamma_{\mathfrak{f}}(\boldsymbol{\alpha}; P)\ll P^{dn}(\log P)^{n+1}X^{-\kappa},$$
or (ii) there exist $a_1,\ldots,a_R\in \Z$ and $q\in \N$ such that
$$(a_1,\ldots,a_R,q)=1, \, q\le X^{R(d-1)} \textrm{ and } |q\alpha_i-a_i|\le X^{R(d-1)}P^{-d}\, \textrm{ for }\, i=1,\ldots,R,$$
or (iii) there exists a real number $c=c_{\mathfrak{f}}>0$ such that
 $$\mathscr{L}(cX)\gg (cX)^{(d-1)n-\kappa}(\log P)^{\frac{1}{2}}.$$
 \end{lemma}

\begin{proof}Suppose that (i) and (iii) fail to hold. We need to prove (ii). By Lemma \ref{lemmaboundgammatoN} and Lemma \ref{lemmaNtoN}, we have
\begin{align*}\Gamma_{\mathfrak{f}}(\boldsymbol{\alpha}; P)\ll P^{dn}(\log P)^{n}X^{-(d-1)n}N(cX,c^{d-1}X^{d-1}P^{-d};\boldsymbol{\alpha}),\end{align*}
where $c>0$ is a small number depending on $\mathfrak{f}$.
Therefore,
\begin{align*}N(cX,c^{d-1}X^{d-1}P^{-d};\boldsymbol{\alpha})\gg (cX)^{(d-1)n-\kappa} (\log P) .\end{align*}
Since (iii) does not hold, there exists an integer point
$$(\mathbf{x}^{(1)},\ldots,\mathbf{x}^{(d-1)})\in \Z^{(d-1)n}$$
counted by $N(cX,c^{d-1}X^{d-1}P^{-d};\boldsymbol{\alpha})$ such that
\begin{align*}\rank\Upsilon(\mathbf{x}^{(1)},\ldots,\mathbf{x}^{(d-1)};\mathfrak{f})=R.\end{align*}
Therefore, $\Upsilon(\mathbf{x}^{(1)},\ldots,\mathbf{x}^{(d-1)};\mathfrak{f})$ has a $R$ by $R$ invertible submatrix. Without loss of generality, we assume
$$\big(\Psi_j^{(i)}(\mathbf{x}^{(1)},\ldots,\mathbf{x}^{(d-1)};\mathfrak{f})\big)_{\substack{1\le i,j\le R}}$$
is invertible, and let $q'$ be the absolute value of its determinant. We have
\begin{align}\label{boundq}q'\le R! C_{\mathfrak{f}}^R (cX)^{R(d-1)},\end{align}
where $C_{\mathfrak{f}}$ denotes the maximum of the absolute values of the coefficients of $\mathfrak{f}$.

 According to the definition of $N(cX,c^{d-1}X^{d-1}P^{-d};\boldsymbol{\alpha})$ in \eqref{defineNXY}, the integer point $(\mathbf{x}^{(1)},\ldots,\mathbf{x}^{(d-1)})$ also satisfies
\begin{align*}|\mathbf{x}^{(1)}|\le cX,\ldots,|\mathbf{x}^{(d-1)}|\le cX\end{align*}
and
\begin{align*}\|\Phi_j(\boldsymbol{\alpha};\mathbf{x}^{(1)},\ldots,\mathbf{x}^{(d-1)};\mathfrak{f})\|<c^{d-1}X^{d-1}P^{-d}\ \textrm{ for all } 1\le j\le R.\end{align*}
In particular, for each $1\le j\le R$, there exist $A_j\in \Z$ and  $\delta_j\in \R$ such that
\begin{align*}\sum_{i=1}^R\alpha_i\Psi_j^{(i)}(\mathbf{x}^{(1)},\ldots,\mathbf{x}^{(d-1)};\mathfrak{f})=A_j+\delta_j \ \textrm{ and }\ |\delta_j|<c^{d-1}X^{d-1}P^{-d}.\end{align*}

Let $a_1',\ldots,a_R'\in \Z$ be the integer solution to
$$\sum_{i=1}^Ra_i'\Psi_j^{(i)}(\mathbf{x}^{(1)},\ldots,\mathbf{x}^{(d-1)};\mathfrak{f})=q'A_j\ \textrm{ for all } 1\le j\le R.$$
Then we have
$$\sum_{i=1}^R(q'\alpha_i-a_i')\Psi_j^{(i)}(\mathbf{x}^{(1)},\ldots,\mathbf{x}^{(d-1)};\mathfrak{f})=q'\delta_j\ \textrm{ for all } 1\le j\le R.$$
We conclude that
\begin{align}\label{bounddif}|q'\alpha_i-a_i'|\le R!C_{\mathfrak{f}}^{R-1}(cX)^{(R-1)(d-1)}c^{d-1}X^{d-1}P^{-d}\ \textrm{ for all } 1\le i\le R.\end{align}

By \eqref{boundq} and \eqref{bounddif}, if $c$ is sufficiently small, then $q'\le X^{R(d-1)}$ and
$$|q'\alpha_i-a_i'|\le X^{R(d-1)}P^{-d}\ \textrm{ for all } 1\le i\le R.$$
We establish (ii) by throwing away the common divisor of $a_1',\ldots,a_R',q'$.
\end{proof}

 \begin{lemma}\label{lemmagamma}Let $\kappa>0$ be a real number. Suppose that $X\le P$. Then we have either (i)
 $$\Gamma_{\mathfrak{f}}(\boldsymbol{\alpha}; P)\ll P^{dn}(\log P)^{n+1}X^{-\kappa},$$
or (ii) there exist $a_1,\ldots,a_R\in \Z$ and $q\in \N$ such that
$$(a_1,\ldots,a_R,q)=1, \ q\le X^{R(d-1)} \textrm{ and } |q\alpha_i-a_i|\le X^{R(d-1)}P^{-d}\ \textrm{ for }\ i=1,\ldots,R,$$
or (iii)
 $$n-\dim V_{\mathfrak{f}}^\ast< \kappa.$$\end{lemma}
 \begin{proof}We prove (iii) by assuming that both (i) and (ii) fail to hold. In view of Lemma \ref{lemmaoneofthree}, we have
$$\mathscr{L}(cX)\gg (cX)^{(d-1)n-\kappa}(\log P)^{\frac12}.$$
By Lemma 3.1 in \cite{Birch},
 $$\mathscr{L}(cX)\ll X^{\dim \mathscr{L}}.$$
Thus we conclude from above
 $$\dim \mathscr{L}> (d-1)n-\kappa.$$
We apply Lemma 3.3 in \cite{Birch} to deduce
 $$\dim V_{\mathfrak{f}}^\ast \ge \dim \mathscr{L}-(d-2)n>n-\kappa,$$
 and therefore,
 $$n-\dim V_{\mathfrak{f}}^\ast < \kappa.$$
 This completes the proof.
 \end{proof}

  \begin{lemma}\label{lemmadif}Let $f_1(\mathbf{x}),\ldots,f_R(\mathbf{x})$ be forms of degree $d\ge 2$ in $n$ variables. Suppose that $G(\mathbf{x})$ is a polynomial of $\mathbf{x}$ with degree smaller than $d$ and the coefficients of $G$ are real numbers. Then we have
  \begin{align*}
  \sum_{1\le \mathbf{x}\le P}e\big(\sum_{i=1}^R\alpha_i f_i(\mathbf{x})+ G(\mathbf{x})\big)\ll P^{n-2^{1-d}dn}\Gamma_{\mathfrak{f}}(\boldsymbol{\alpha}; P)^{2^{1-d}},\end{align*}
  where the implied constant is independent of $G$.
  \end{lemma}
 \begin{proof}This is essentially Lemma 2.1 of Birch \cite{Birch}, and it can be proved by the standard difference argument.
 Note that $G$ disappears in this process, since the degree of $G$ is smaller than $d$.\end{proof}

 Lemma \ref{lemmagamma} and Lemma \ref{lemmadif} together lead to the following result.
  \begin{lemma}\label{lemmaboundS}Let $\kappa>0$ be a real number. Let $f_1,\ldots,f_R$ and $G$ be as in Lemma \ref{lemmadif}. Suppose that $Q\le \frac{1}{4}P^{\frac{Rd}{R+1}}$. Then we have either (i)
  \begin{align*}
  \sum_{1\le \mathbf{x}\le P}e\big(\sum_{i=1}^R\alpha_i f_i(\mathbf{x})+ G(\mathbf{x})\big)\ll P^{n}(\log P)^{n}Q^{-\frac{\kappa}{2^{d-1}(d-1)R}},\end{align*}
or (ii) there exist $a_1,\ldots,a_R\in \Z$ and $q\in \N$ such that
$$(a_1,\ldots,a_R,q)=1, \ q\le Q\ \textrm{ and }\ |q\alpha_i-a_i|\le QP^{-d}\ \textrm{ for }\ i=1,\ldots,R,$$
or (iii)
 $$n-\dim V_{\mathfrak{f}}^\ast<\kappa.$$\end{lemma}
\begin{proof}Note that
$$Q^{\frac{1}{R(d-1)}}\le P.$$
We can obtain Lemma \ref{lemmaboundS} by applying Lemma \ref{lemmagamma} with $X=Q^{\frac{1}{R(d-1)}}$ and Lemma \ref{lemmadif}.\end{proof}

\vskip3mm

\section{Mean value theorems}

Suppose that $h_i(\mathbf{x},\mathbf{w})$ is a polynomial of $(\mathbf{x},\mathbf{w})$, and the degree of $h_i$ regard to $\mathbf{x}$ is smaller than $d$, that is,
\begin{align*}\deg_{\mathbf{x}}(h_i)<d.\end{align*}
Let $\mathbf{h}=(h_1,\ldots,h_R)$. For $\boldsymbol{\alpha}\in \R^{R}$ and $\mathbf{x}\in \Z^{k}$, we define
\begin{align*}\mathcal{E}(\boldsymbol{\alpha};\mathbf{x})=\sum_{1\le \mathbf{w}\le P}\lambda(\mathbf{w})e\big(\boldsymbol{\alpha}\cdot \mathbf{h}(\mathbf{x},\mathbf{w})\big),\end{align*}
where the summation is taken over $t$-dimensional vectors in $\N^{t}$. Similarly, we introduce
\begin{align*}T(\boldsymbol{\alpha};\mathbf{x})=\sum_{1\le \mathbf{u}\le P}\lambda(\mathbf{u})e\big(\boldsymbol{\alpha}\cdot \mathbf{H}(\mathbf{x},\mathbf{u})\big),\end{align*}
where the summation is taken over $l$-dimensional vectors in $\N^{l}$, $H_i(\mathbf{x},\mathbf{w})$ is a polynomial of $(\mathbf{x},\mathbf{u})$, and the degree of $H_i$ regard to $\mathbf{x}$ is smaller than $d$, that is,
\begin{align*}\deg_{\mathbf{x}}(H_i)<d.\end{align*}
Note that when $t=0$, we view $\mathcal{E}(\boldsymbol{\alpha};\mathbf{x})$ as a constant function $\mathcal{E}\equiv 1$.

We are interested in the second moment of the following
\begin{align*}\int_{\mathfrak{n}}e\big(\boldsymbol{\alpha}\cdot \mathbf{g}(\mathbf{x})\big)\mathcal{E}(\boldsymbol{\alpha};\mathbf{x})T(\boldsymbol{\alpha};\mathbf{x})d\boldsymbol{\alpha},\end{align*}
where $\mathbf{g}=(g_1,\ldots,g_R)$ are forms of degree $d$ and $\mathfrak{n}\subset [0,2]^R$ is a measurable set. In our applications, the set
$\mathfrak{n}$ will be the minor arcs.  Precisely, we are interested in
\begin{align}\label{defineJ}\mathcal{J}:=\mathcal{J}_{\mathfrak{n}} =\sum_{1\le \mathbf{x}\le P}\Big|\int_{\mathfrak{n}}e\big(\boldsymbol{\alpha} \cdot \mathbf{g}(\mathbf{x})\big)\mathcal{E}(\boldsymbol{\alpha};\mathbf{x})
T(\boldsymbol{\alpha};\mathbf{x})d\boldsymbol{\alpha}\Big|^2.\end{align}

In the following result, we obtain an upper bound of $\mathcal{J}$.
\begin{lemma}\label{lemmameanJ}Let $\mathcal{J}$ be defined in \eqref{defineJ}. Let $\kappa>0$ be a real number satisfying
$\kappa>2^{d-1}(d-1)R(R+1)$. Suppose that
$$k-\dim V_{\mathbf{g}}^\ast \ge \kappa.$$
Let $X\le \frac{1}{4}P^{\frac{Rd}{R+1}}$. Then we have
\begin{align*}\mathcal{J}\ll &\, P^{k+2t+2l-Rd}(\log P)^{k+2t+2l}X^{R+1-\frac{\kappa}{2^{d-1}(d-1)R}}\big|\mathfrak{n}\big|\notag
\\  &\ \ \ \ \  +X^{R+1}P^{-Rd}\sum_{1\le \mathbf{x}\le P}\int_{\mathfrak{n}}|\mathcal{E}(\boldsymbol{\alpha};\mathbf{x})T(\boldsymbol{\alpha};\mathbf{x})|^2d\boldsymbol{\alpha},\end{align*}
where the implied constant depends only on (the coefficients of) $\mathbf{g}$, and we use $|\mathfrak{n}|$ to denote the measure of $\mathfrak{n}$, that is,
$$|\mathfrak{n}|=\int_{\mathfrak{n}}d\alpha.$$
\end{lemma}
\begin{proof}On expanding the square, we obtain
\begin{align*}\mathcal{J}=  \sum_{\mathbf{x}}\int_{\mathfrak{n}}\int_{\mathfrak{n}}e\big((\boldsymbol{\alpha}_1-\boldsymbol{\alpha}_2)\cdot \mathbf{g}(\mathbf{x})\big)\mathcal{H}(\boldsymbol{\alpha}_1,\boldsymbol{\alpha}_2,\mathbf{x})
d\boldsymbol{\alpha}_1 d\boldsymbol{\alpha}_2,\end{align*}
where
\begin{align*}\mathcal{H}(\boldsymbol{\alpha}_1,\boldsymbol{\alpha}_2,\mathbf{x})=
\mathcal{E}(\boldsymbol{\alpha}_1;\mathbf{x})\mathcal{E}(-\boldsymbol{\alpha}_2;\mathbf{x})
T(\boldsymbol{\alpha}_1;\mathbf{x})T(-\boldsymbol{\alpha}_2;\mathbf{x})
.\end{align*}
On interchanging the order of the summation and integrations, we further obtain
\begin{align*}\mathcal{J}=  \int_{\mathfrak{n}}\int_{\mathfrak{n}}\mathcal{G}(\boldsymbol{\alpha}_1,\boldsymbol{\alpha}_2)d\boldsymbol{\alpha}_1 d\boldsymbol{\alpha}_2,\end{align*}
where
\begin{align}\label{forboundGtrivial}\mathcal{G}(\boldsymbol{\alpha}_1,\boldsymbol{\alpha}_2)= \sum_{\mathbf{x}}e\big((\boldsymbol{\alpha}_1-\boldsymbol{\alpha}_2) \cdot \mathbf{g}(\mathbf{x})\big)\mathcal{H}(\boldsymbol{\alpha}_1,\boldsymbol{\alpha}_2,\mathbf{x}).\end{align}

By the definitions of $\mathcal{E}(\boldsymbol{\alpha};\mathbf{x})$ and $T(\boldsymbol{\alpha};\mathbf{x})$, we have
\begin{align*}\mathcal{H}(\boldsymbol{\alpha}_1,\boldsymbol{\alpha}_2,\mathbf{x})=
\sum_{\mathbf{w}_1}\sum_{\mathbf{w}_2}\sum_{\mathbf{u}_1}\sum_{\mathbf{u}_2}
\lambda(\mathbf{w}_1)\lambda(\mathbf{w}_2)\lambda(\mathbf{u}_1)\lambda(\mathbf{u}_2)e\big(p(\mathbf{x})\big),\end{align*}
where $p(\mathbf{x}):=p_{\boldsymbol{\alpha}_1,\boldsymbol{\alpha}_2}(\mathbf{x},\mathbf{w}_1,
\mathbf{w}_2,\mathbf{u}_1,
\mathbf{u}_2)$ is
\begin{align*}\boldsymbol{\alpha}_1 \cdot \mathbf{h}(\mathbf{x},\mathbf{w}_1)
-\boldsymbol{\alpha}_2 \cdot \mathbf{h}(\mathbf{x},\mathbf{w}_2)+\boldsymbol{\alpha}_1 \cdot \mathbf{H}(\mathbf{x},\mathbf{u}_1)
-\boldsymbol{\alpha}_2 \cdot \mathbf{H}(\mathbf{x},\mathbf{u}_2)
.\end{align*}
On interchanging the order of summations, we deduce that
\begin{align*}\mathcal{G}(\boldsymbol{\alpha}_1,\boldsymbol{\alpha}_2)=\sum_{\mathbf{w}_1}\sum_{\mathbf{w}_2}\sum_{\mathbf{u}_1}\sum_{\mathbf{u}_2}
\lambda(\mathbf{w}_1)\lambda(\mathbf{w}_2)\lambda(\mathbf{u}_1)\lambda(\mathbf{u}_2)
U(\boldsymbol{\alpha}_1,\boldsymbol{\alpha}_2),\end{align*}
where $U(\boldsymbol{\alpha}_1,\boldsymbol{\alpha}_2):=U(\boldsymbol{\alpha}_1,\boldsymbol{\alpha}_2,\mathbf{w}_1,\mathbf{w}_2,\mathbf{u}_1,
\mathbf{u}_2)$ denotes
\begin{align*}U(\boldsymbol{\alpha}_1,\boldsymbol{\alpha}_2,\mathbf{w}_1,\mathbf{w}_2,\mathbf{u}_1,
\mathbf{u}_2)=\sum_{\mathbf{x}}e\big((\boldsymbol{\alpha}_1-\boldsymbol{\alpha}_2) \cdot \mathbf{g}(\mathbf{x})+p(\mathbf{x})\big).\end{align*}
Note that as a polynomial of $\mathbf{x}$, the degree of $p(\mathbf{x})$ is smaller than $d$.

For $\boldsymbol{\alpha}_1-\boldsymbol{\alpha}_2\in \mathfrak{m}(Q)$ with $Q\le \frac{1}{4}P^{\frac{Rd}{R+1}}$, we conclude from Lemma \ref{lemmaboundS} that
\begin{align*} U(\boldsymbol{\alpha}_1,\boldsymbol{\alpha}_2,\mathbf{x},\mathbf{w}_1,\mathbf{w}_2,\mathbf{u}_1,
\mathbf{u}_2)
\ll  P^{k}(\log P)^{k}Q^{-\frac{\kappa}{2^{d-1}(d-1)R}},\end{align*}
and therefore,
\begin{align}\label{boundGcal} \mathcal{G}(\boldsymbol{\alpha}_1,\boldsymbol{\alpha}_2)
\ll  P^{k+2t+2l}(\log P)^{k+2t+2l}Q^{-\frac{\kappa}{2^{d-1}(d-1)R}}.\end{align}

\medskip

For a measurable set $\mathcal{M}\subseteq [-1,1]^R$, we introduce
\begin{align*}K(\mathcal{M},\boldsymbol{\alpha})=\begin{cases}
1, \ \ \textrm{ if } \boldsymbol{\alpha}\in \mathcal{M},
\\0, \ \ \textrm{ otherwise}.\end{cases}\end{align*}
Then we define
\begin{align*}\mathcal{J}(\mathcal{M})=  \int_{\mathfrak{n}}\int_{\mathfrak{n}}\mathcal{G}(\boldsymbol{\alpha}_1,\boldsymbol{\alpha}_2)
K(\mathcal{M},\boldsymbol{\alpha}_1-\boldsymbol{\alpha}_2)d\boldsymbol{\alpha}_1 d\boldsymbol{\alpha}_2.\end{align*}

Now we define $\mathfrak{M}'(Q)$ similarly to \eqref{defineMQ} with $1\le a_1,\ldots,a_R\le q$ replaced by $-2q\le a_1,\ldots,a_R\le q$. Then we define $\mathfrak{m}'(Q)=[-1,1]^{R}\setminus \mathfrak{M}'(Q)$.
In fact, $\mathfrak{M}'(Q)$ and $\mathfrak{m}'(Q)$ coincide with $\mathfrak{M}(Q)$ and $\mathfrak{m}(Q)$ respectively in the sense of modulo one. And we shall just write $\mathfrak{M}(Q)$ and $\mathfrak{m}(Q)$ instead of $\mathfrak{M}'(Q)$ and $\mathfrak{m}'(Q)$ for simplicity.

Let
$$\mathcal{M}_1=\mathfrak{m}(\frac{1}{4}P^{\frac{Rd}{R+1}}).$$
 We obtain from  \eqref{boundGcal} that
\begin{align}\label{JB1}\mathcal{J}(\mathcal{M}_1) \ll P^{k+2t+2l}(\log P)^{k+2t+2l}P^{-\frac{d\kappa}{2^{d-1}(d-1)(R+1)}}|\mathfrak{n}|.\end{align}
For $X\le Q\le \frac{1}{8}P^{\frac{Rd}{R+1}}$, let
$$\mathcal{M}_2:=\mathcal{M}_2(Q)=\mathfrak{M}(2Q)\setminus \mathfrak{M}(Q).$$
 We appeal to \eqref{boundGcal} again to obtain
\begin{align}\label{JB2}\mathcal{J}(\mathcal{M}_2) \ll P^{k+2t+2l}(\log P)^{k+2t+2l}Q^{-\frac{\kappa}{2^{d-1}(d-1)R}}
\int_{\mathfrak{n}}\int_{\mathfrak{n}}K(\mathcal{M}_2,\boldsymbol{\alpha}_1-\boldsymbol{\alpha}_2)d\boldsymbol{\alpha}_1 d\boldsymbol{\alpha}_2.\end{align}
Since $\mathcal{M}_2\subseteq \mathfrak{M}(2Q)$, we have uniformly for $\boldsymbol{\alpha}_1$ that
\begin{align*}\int_{\mathfrak{n}}K(\mathcal{M}_2,\boldsymbol{\alpha}_1-\boldsymbol{\alpha}_2) d\boldsymbol{\alpha}_2 \ll Q^{R+1}P^{-Rd}.\end{align*}
Then we have
\begin{align}\label{trivialB2T}
\int_{\mathfrak{n}}\int_{\mathfrak{n}}K(\mathcal{M}_2,\boldsymbol{\alpha}_1-\boldsymbol{\alpha}_2)d\boldsymbol{\alpha}_1 d\boldsymbol{\alpha}_2\ll Q^{R+1}P^{-Rd}|\mathfrak{n}|.\end{align}
Putting  \eqref{trivialB2T} into \eqref{JB2}, we arrive at
\begin{align}\label{JB2bound}\mathcal{J}(\mathcal{M}_2) \ll  P^{k+2t+2l-Rd}(\log P)^{k+2t+2l}Q^{R+1-\frac{\kappa}{2^{d-1}(d-1)R}}|\mathfrak{n}|.\end{align}
By the standard dyadic argument, we obtain from  \eqref{JB1} and \eqref{JB2bound} together with $\kappa>2^{d-1}(d-1)R(R+1)$ that
\begin{align}\label{Jm}\mathcal{J}(\mathfrak{m}) \ll P^{k+2t+2l-Rd}(\log P)^{k+2t+2l}X^{R+1-\frac{\kappa}{2^{d-1}(d-1)R}}|\mathfrak{n}|,\end{align}
where
\begin{align*}\mathfrak{m}=\mathfrak{m}(X).\end{align*}

Now we deal with $\mathcal{J}(\mathfrak{M})$, where
$$\mathfrak{M}=\mathfrak{M}(X).$$
On applying the trivial bound to the summation over $\mathbf{x}$ in \eqref{forboundGtrivial}, we obtain
\begin{align*}|\mathcal{G}(\boldsymbol{\alpha}_1,\boldsymbol{\alpha}_2)|\le \sum_{\mathbf{x}}|\mathcal{H}(\boldsymbol{\alpha}_1,\boldsymbol{\alpha}_2,\mathbf{x})|.\end{align*}
Therefore, we have
\begin{align*}|\mathcal{J}(\mathfrak{M})| \le  \sum_{\mathbf{x}}\int_{\mathfrak{n}}\int_{\mathfrak{n}}|\mathcal{H}(\boldsymbol{\alpha}_1,\boldsymbol{\alpha}_2,\mathbf{x})|K(\mathfrak{M},\boldsymbol{\alpha}_1-\boldsymbol{\alpha}_2)
d\boldsymbol{\alpha}_1 d\boldsymbol{\alpha}_2.\end{align*}
Using the inequality
\begin{align*}2|\mathcal{H}(\boldsymbol{\alpha}_1,\boldsymbol{\alpha}_2,\mathbf{x})|
 \le  |\mathcal{E}(\boldsymbol{\alpha}_1;\mathbf{x})T(\boldsymbol{\alpha}_1;\mathbf{x})|^2+|\mathcal{E}(\boldsymbol{\alpha}_2;\mathbf{x})
T(\boldsymbol{\alpha}_2;\mathbf{x})|^2,\end{align*}
one can further deduce  by symmetry that
\begin{align*}|\mathcal{J}(\mathfrak{M})| \le  \sum_{\mathbf{x}}\int_{\mathfrak{n}}\int_{\mathfrak{n}}|\mathcal{E}(\boldsymbol{\alpha}_1;\mathbf{x})T(\boldsymbol{\alpha}_1;\mathbf{x})|^2
K(\mathfrak{M},\boldsymbol{\alpha}_1-\boldsymbol{\alpha}_2)d\boldsymbol{\alpha}_1 d\boldsymbol{\alpha}_2.\end{align*}
Note that
\begin{align*}\int_{\mathfrak{n}}K(\mathfrak{M},\boldsymbol{\alpha}_1-\boldsymbol{\alpha}_2) d\boldsymbol{\alpha}_2 \ll X^{R+1}P^{-Rd}\end{align*}
uniformly for $\boldsymbol{\alpha}_1$. Therefore,
\begin{align}\label{JM}|\mathcal{J}(\mathfrak{M})| \ll  X^{R+1}P^{-Rd}\sum_{\mathbf{x}}\int_{\mathfrak{n}}|\mathcal{E}(\boldsymbol{\alpha};\mathbf{x})T(\boldsymbol{\alpha};\mathbf{x})|^2d\boldsymbol{\alpha}.
\end{align}
We complete the proof by combining \eqref{Jm} and \eqref{JM}.
\end{proof}

We explain a special case of Lemma \ref{lemmameanJ}. Suppose that $G_i(\mathbf{x},\mathbf{y})$ is polynomial of $(\mathbf{x},\mathbf{y})$, and the degree of $G_i$ as a polynomial of $\mathbf{x}$ is smaller than $d$, that is,
\begin{align*}\deg_{\mathbf{x}}(G_i)<d.\end{align*} For $\mathbf{x}\in \Z^{m}$, we introduce
\begin{align*}T_0(\boldsymbol{\alpha};\mathbf{x})=\sum_{1\le \mathbf{y}\le P}\lambda(\mathbf{y})e\big(\boldsymbol{\alpha} \cdot \mathbf{G}(\mathbf{x},\mathbf{y})\big),\end{align*}
where the summation is taken over $l$-dimensional vectors in $\N^{l}$. Let
\begin{align}\label{defineI}\mathcal{I}:=\mathcal{I}_\mathfrak{n}=\sum_{1\le \mathbf{x}\le P}\Big|\int_{\mathfrak{n}}
e\big(\boldsymbol{\alpha} \cdot \mathbf{f}(\mathbf{x})\big)T_0(\boldsymbol{\alpha};\mathbf{x})d\boldsymbol{\alpha}\Big|^2,\end{align}
where $\mathbf{f}=(f_1,\ldots,f_R)$ and $f_1,\ldots,f_R$ are forms of degree $d$.

As a special case of Lemma \ref{lemmameanJ}, we have the following.
\begin{lemma}\label{lemmameanI}Let $\mathcal{I}$ be given in \eqref{defineI}. Let $\kappa>0$ be a real number satisfying
$\kappa>2^{d-1}(d-1)R(R+1)$. Suppose that
$$m-\dim V_{\mathbf{f}}^\ast \ge \kappa.$$
Let $X\le \frac{1}{4}P^{\frac{Rd}{R+1}}$. Then we have
\begin{align*}\mathcal{I}\ll &\, P^{m+2l-Rd}(\log P)^{m+2l}X^{R+1-\frac{\kappa}{2^{d-1}(d-1)R}}|\mathfrak{n}|
\\ & \ \ \ \ \ \ \ \ \ \ +X^{R+1}P^{-Rd}\sum_{1\le \mathbf{x}\le P}\int_{\mathfrak{n}}\big|T_0(\boldsymbol{\alpha};\mathbf{x})\big|^2d\boldsymbol{\alpha}.\end{align*}
\end{lemma}

One can obtain Lemma \ref{lemmameanI} from Lemma \ref{lemmameanJ} by choosing $k=m$, $t=0$, $\mathcal{E}\equiv 1$ and $T(\boldsymbol{\alpha};\mathbf{x})=T_0(\boldsymbol{\alpha};\mathbf{x})$.

 \vskip3mm

\section{A crucial proposition}

Suppose that $\mathbf{F}=(F_1,\ldots,F_R)$ and $F_i=F_i(x_1,\ldots,x_n)$ are forms of degree $d$ in $n$ variables. We write $$\mathbf{x}=(\mathbf{y},\mathbf{z},\mathbf{w}),$$
 where $\mathbf{y}\in \N^m, \mathbf{z}\in \N^{s}, \mathbf{w}\in \N^t$ and
 $$m+s+t=n.$$
  Then each $F_i$ can be uniquely represented as
\begin{align}\label{representFintofgh1}F_i(\mathbf{y},\mathbf{z},\mathbf{w})=f_i(\mathbf{y})+g_i(\mathbf{y},\mathbf{z})
+h_i(\mathbf{y},\mathbf{z},\mathbf{w}),\end{align}
where the degree of $g_i$ as a polynomial of $\mathbf{y}$ is smaller than $d$ and the degree of $h_i$ as a polynomial of $(\mathbf{y},\mathbf{z})$ is also smaller than $d$, that is,
\begin{align*}\deg_{\mathbf{y}}(g_i)<d \ \textrm{ and } \ \deg_{(\mathbf{y},\mathbf{z})}(h_i)<d.\end{align*}
Let
\begin{align*}\mathbf{f}=(f_1,\ldots,f_R),\ \mathbf{g}=(g_1,\ldots,g_R)\ \textrm{ and } \ \mathbf{h}=(h_1,\ldots,h_R).\end{align*}

The purpose of this section is to find a mean value estimate for $S_\mathbf{F}(\boldsymbol{\alpha})$, which now can be represented as
\begin{align}\label{defineSFinsecttrans}S_{\mathbf{F}}(\boldsymbol{\alpha})=\sum_{1\le \mathbf{y}\le P}\sum_{1\le \mathbf{z}\le P}
\sum_{1\le \mathbf{w}\le P}\lambda(\mathbf{y})
\lambda(\mathbf{z})\lambda(\mathbf{w})e\big(\boldsymbol{\alpha}\cdot \mathbf{F}(\mathbf{y},\mathbf{z},\mathbf{w})\big).\end{align}
We define
\begin{align}\label{writeRsecttrans}\mathcal{E}_{\mathbf{y},\mathbf{z}}(\boldsymbol{\alpha})=
\sum_{1\le \mathbf{w}\le P}\lambda(\mathbf{w})e\big(\boldsymbol{\alpha}\cdot \mathbf{h}(\mathbf{y},\mathbf{z},\mathbf{w})\big).\end{align}
\begin{proposition}\label{prop}
Let $S_{\mathbf{F}}(\boldsymbol{\alpha})$ be given in \eqref{defineSFinsecttrans} with $\mathbf{F}=(F_1,\ldots,F_R)$ being (uniquely) represented as in \eqref{representFintofgh1}. Let $X\le\frac{1}{4}P^{\frac{Rd}{R+1}}$. Let $\kappa_1$ and $\kappa_2$ be two real numbers satisfying
$\min(\kappa_1,\kappa_2)>2^{d-1}(d-1)R(R+1)$.
Suppose that
\begin{align*}m-\dim V_{\mathbf{f}}^\ast \ge \kappa_1 \ \textrm{ and }\ m+s-\dim V_{\mathbf{g}}^\ast\ge \kappa_2. \end{align*}
Then we have
\begin{align*}\int_{\mathfrak{n}}S_{\mathbf{F}}(\boldsymbol{\alpha})d\boldsymbol{\alpha} \ll\, &P^{n-\frac{1}{2}Rd}(\log P)^{2n}
X^{\frac{1}{2}(R+1)-\frac{\kappa_1}{2^{d}(d-1)R}}|\mathfrak{n}|^{1/2}
\\ & +P^{n-\frac{3}{4}Rd}(\log P)^{2n}X^{\frac{3}{4}(R+1)-\frac{\kappa_2}{2^{d+1}(d-1)R}}|\mathfrak{n}|^{1/4}
\\ & +P^{m+s-Rd}(\log P)^{m+s}X^{R+1}\sup(\mathcal{E}),\end{align*}
where
\begin{align*}\sup(\mathcal{E})=\sup_{\boldsymbol{\alpha}\in \mathfrak{n}}\sup_{\mathbf{y}}\sup_{\mathbf{z}}|\mathcal{E}_{\mathbf{y},\mathbf{z}}(\boldsymbol{\alpha})|.\end{align*}
\end{proposition}
\begin{proof}We introduce
\begin{align}\label{defineTinsect5}T(\boldsymbol{\alpha};\mathbf{y})=
\sum_{\mathbf{z}}\sum_{\mathbf{w}}\lambda(\mathbf{z})\lambda(\mathbf{w})e\big(\boldsymbol{\alpha}\cdot \mathbf{G}(\mathbf{y},\mathbf{z},\mathbf{w})\big),\end{align}
where
\begin{align}\label{defineGinsect5}\mathbf{G}(\mathbf{y},\mathbf{z},\mathbf{w})=
\mathbf{g}(\mathbf{y},\mathbf{z})+\mathbf{h}(\mathbf{y},\mathbf{z},\mathbf{w}).\end{align}
By \eqref{representFintofgh1} and \eqref{defineSFinsecttrans}, we have
\begin{align*}S_{\mathbf{F}}(\boldsymbol{\alpha})=\sum_{\mathbf{y}}\lambda(\mathbf{y})
e\big(\boldsymbol{\alpha} \cdot \mathbf{f}(\mathbf{y})\big)T(\boldsymbol{\alpha};\mathbf{y}),\end{align*}
and therefore,
\begin{align*}\int_{\mathfrak{n}}S_{\mathbf{F}}(\boldsymbol{\alpha})d\boldsymbol{\alpha}=\int_{\mathfrak{n}}\sum_{\mathbf{y}}\lambda(\mathbf{y})
e\big(\boldsymbol{\alpha} \cdot \mathbf{f}(\mathbf{y})\big)T(\boldsymbol{\alpha};\mathbf{y})d\boldsymbol{\alpha}.\end{align*}
Then we interchange the order of the integration and the summation to deduce that
\begin{align*}\int_{\mathfrak{n}}S_{\mathbf{F}}(\boldsymbol{\alpha})d\boldsymbol{\alpha}=\sum_{\mathbf{y}}\lambda(\mathbf{y})\int_{\mathfrak{n}}
e\big(\boldsymbol{\alpha} \cdot \mathbf{f}(\mathbf{y})\big)T(\boldsymbol{\alpha};\mathbf{y})d\boldsymbol{\alpha}.\end{align*}

By Cauchy's inequality,
\begin{align}\label{SsquaretoI}\Big|\int_{\mathfrak{n}}S_{\mathbf{F}}(\boldsymbol{\alpha})d\boldsymbol{\alpha}\Big|^2\le P^m(\log P)^{2m}\,\mathcal{I}_{\mathfrak{n}},\end{align}
where
\begin{align*}\mathcal{I}_{\mathfrak{n}}=\sum_{\mathbf{y}}\Big|\int_{\mathfrak{n}}
e\big(\boldsymbol{\alpha}\cdot \mathbf{f}(\mathbf{y})\big)T(\boldsymbol{\alpha};\mathbf{y})d\boldsymbol{\alpha}\Big|^2.\end{align*}
By Lemma \ref{lemmameanI},
\begin{align}\label{ItoT}\mathcal{I}_{\mathfrak{n}}\ll P^{m+2l-Rd}(\log P)^{m+2l}X^{R+1-\frac{\kappa_1}{2^{d-1}R(d-1)}}|\mathfrak{n}|+X^{R+1}P^{-Rd}\mathcal{T}_{\mathfrak{n}},\end{align}
where
$$l=s+t$$
and
\begin{align}\label{defineTn}\mathcal{T}_{\mathfrak{n}} =\sum_{\mathbf{y}}\int_{\mathfrak{n}}\big|T(\boldsymbol{\alpha};\mathbf{y})\big|^2d\boldsymbol{\alpha}.\end{align}

Now we deal with $\mathcal{T}_{\mathfrak{n}}$. By \eqref{writeRsecttrans}, \eqref{defineTinsect5} and \eqref{defineGinsect5}, we have
\begin{align*}T(\boldsymbol{\alpha};\mathbf{y})=\sum_{\mathbf{z}}\lambda(\mathbf{z})e\big(\boldsymbol{\alpha} \cdot \mathbf{g}(\mathbf{y},\mathbf{z})\big)
\mathcal{E}_{\mathbf{y},\mathbf{z}}(\boldsymbol{\alpha}).\end{align*}
We deduce that
\begin{align*}\int_{\mathfrak{n}}\big|T(\boldsymbol{\alpha};\mathbf{y})\big|^2d\boldsymbol{\alpha}=&\int_{\mathfrak{n}}T(\boldsymbol{\alpha};\mathbf{y})
T(-\boldsymbol{\alpha};\mathbf{y})d\boldsymbol{\alpha}\\ = & \sum_{\mathbf{z}}\lambda(\mathbf{z})\int_{\mathfrak{n}}e\big(\boldsymbol{\alpha} \cdot \mathbf{g}(\mathbf{y},\mathbf{z})\big)
\mathcal{E}_{\mathbf{y},\mathbf{z}}(\boldsymbol{\alpha})
T(-\boldsymbol{\alpha};\mathbf{y})d\boldsymbol{\alpha}.\end{align*}
By Cauchy's inequality,
\begin{align}\label{TsquaretoJ1}\Big(\int_{\mathfrak{n}}\big|T(\boldsymbol{\alpha};\mathbf{y})\big|^2d\boldsymbol{\alpha}\Big)^2 \le P^s(\log P)^{2s}\,\mathcal{J}_{\mathfrak{n},\mathbf{y}},\end{align}
where
\begin{align*}\mathcal{J}_{\mathfrak{n},\mathbf{y}}=\sum_{\mathbf{z}}\Big|\int_{\mathfrak{n}}e\big(\boldsymbol{\alpha}\cdot \mathbf{g}(\mathbf{y},\mathbf{z})\big)
\mathcal{E}_{\mathbf{y},\mathbf{z}}(\boldsymbol{\alpha})
T(-\boldsymbol{\alpha};\mathbf{y})d\boldsymbol{\alpha}\Big|^2.\end{align*}
Then we conclude from \eqref{defineTn} and \eqref{TsquaretoJ1} that
\begin{align}\label{TsquaretoJ}\mathcal{T}_{\mathfrak{n}}^2\le P^{m+s}(\log P)^{2s}\,\mathcal{J}_{\mathfrak{n}},\end{align}
where
\begin{align*}\mathcal{J}_{\mathfrak{n}}=\sum_{\mathbf{y}}\mathcal{J}_{\mathfrak{n},\mathbf{y}}.\end{align*}

On applying Lemma \ref{lemmameanJ} with $k=m+s$, $l=s+t$, we obtain
\begin{align}\label{JtoK}\mathcal{J}_{\mathfrak{n}}\ll &\, P^{n+t+2l-Rd}(\log P)^{n+t+2l}X^{R+1-\frac{\kappa_2}{2^{d-1}(d-1)R}}|\mathfrak{n}|
+X^{R+1}P^{-Rd}\mathcal{K}_{\mathfrak{n}},\end{align}
where
\begin{align*}\mathcal{K}_{\mathfrak{n}}=\sum_{\mathbf{y}}\sum_{\mathbf{z}}\int_{\mathfrak{n}}|T(\boldsymbol{\alpha};\mathbf{y})|^2
|\mathcal{E}_{\mathbf{y},\mathbf{z}}(\boldsymbol{\alpha})|^2d\boldsymbol{\alpha}.\end{align*}
On recalling the definition of $\sup(\mathcal{E})$, we have
\begin{align*}\sum_{\mathbf{z}}|\mathcal{E}_{\mathbf{y},\mathbf{z}}(\boldsymbol{\alpha})|^2\le P^{s}\sup(\mathcal{E})^2,\end{align*}
and  we further obtain
\begin{align}\label{Kbound}\mathcal{K}_{\mathfrak{n}}\le \mathcal{T}_{\mathfrak{n}}\, P^{s}\sup(\mathcal{E})^2.\end{align}

We put \eqref{Kbound} into \eqref{JtoK} to obtain
\begin{align*}\mathcal{J}_{\mathfrak{n}}\ll &\, P^{n+t+2l-Rd}(\log P)^{n+t+2l}X^{R+1-\frac{\kappa_2}{2^{d-1}(d-1)R}}|\mathfrak{n}|
\\ & \ \ \ \ \ +
\mathcal{T}_{\mathfrak{n}}\, P^{s-Rd}X^{R+1} \sup(\mathcal{E})^2,\end{align*}
and then we have by \eqref{TsquaretoJ},
\begin{align*}\mathcal{T}_{\mathfrak{n}} ^2\ll\, & P^{2n+2l-Rd}(\log P)^{n+s+3l}X^{R+1-\frac{\kappa_2}{2^{d-1}(d-1)R}}|\mathfrak{n}|
\\ \ &\ \ \ \ \ \ \ \ \ \ \ \ \ \ \ \ +
\mathcal{T}_{\mathfrak{n}}\, P^{m+2s-Rd}(\log P)^{2s}X^{R+1}\sup(\mathcal{E})^2.\end{align*}
Therefore, we  conclude
\begin{align}\label{Tbound}\mathcal{T}_{\mathfrak{n}} \ll &\,P^{n+l-\frac{1}{2}Rd}(\log P)^{\frac{1}{2}(n+s+3l)}X^{\frac{1}{2}(R+1)-\frac{\kappa_2}{2^{d}(d-1)R}}|\mathfrak{n}|^{1/2}\notag
\\ & \ \ \ +
P^{m+2s-Rd}(\log P)^{2s}X^{R+1}\sup(\mathcal{E})^2.\end{align}
With the estimate \eqref{Tbound}, we deduce from \eqref{ItoT} that
\begin{align*}\mathcal{I}_{\mathfrak{n}}\ll & P^{m+2l-Rd}(\log P)^{m+2l}X^{R+1-\frac{\kappa_1}{2^{d-1}(d-1)R}}|\mathfrak{n}|
\\ &+P^{n+l-\frac{3}{2}Rd}(\log P)^{\frac{1}{2}(n+s+3l)}X^{\frac{3}{2}(R+1)-\frac{\kappa_2}{2^{d}(d-1)R}}|\mathfrak{n}|^{1/2}
\\ & +
P^{m+2s-2Rd}(\log P)^{2s}X^{2R+2}\sup(\mathcal{E})^2.\end{align*}
Finally, by \eqref{SsquaretoI}, we have
\begin{align*}\Big|\int_{\mathfrak{n}}S_{\mathbf{F}}(\boldsymbol{\alpha})d\boldsymbol{\alpha}\Big|^2\ll \,&P^{2n-Rd}(\log P)^{3n}X^{R+1-\frac{\kappa_1}{2^{d-1}(d-1)R}}|\mathfrak{n}|
\\ &+P^{2n-\frac{3}{2}Rd}(\log P)^{3n}X^{\frac{3}{2}(R+1)-\frac{\kappa_2}{2^{d}(d-1)R}}|\mathfrak{n}|^{1/2}
\\ & +
P^{2m+2s-2Rd}(\log P)^{2m+2s}X^{2R+2}\sup(\mathcal{E})^2.\end{align*}
This completes the proof.
\end{proof}

We give a remark to Proposition \ref{prop}. We shall apply Proposition \ref{prop} with
$$\mathfrak{n}=\mathfrak{M}(2Q)\setminus \mathfrak{M}(Q).$$
Note that $\mathfrak{n}\ll Q^{R+1}P^{-Rd}$. We conclude from Proposition \ref{prop} that
\begin{align*}\int_{\mathfrak{n}}S_\mathbf{F}(\boldsymbol{\alpha})d\boldsymbol{\alpha} \ll\, &P^{n-Rd}(\log P)^{2n}X^{\frac{1}{2}(R+1)-\frac{\kappa_1}{2^{d}(d-1)R}}Q^{\frac{1}{2}(R+1)}
\\ &\ \ +P^{n-Rd}(\log P)^{2n}X^{\frac{3}{4}(R+1)-\frac{\kappa_2}{2^{d+1}(d-1)R}}Q^{\frac{1}{4}(R+1)}
\\ &\ \ +
P^{n-Rd}(\log P)^{n}X^{R+1}\sup(\mathcal{E}).\end{align*}
A trivial bound of $\sup(\mathcal{E})$ is
$$\sup(\mathcal{E})\ll P^{t}(\log P)^t.$$
 Suppose that one has a nontrivial estimate
\begin{align}\label{nontrivialE}\sup(\mathcal{E}) \ll P^{t}(\log P)^t Q^{-\omega_{d,R}} \ \textrm{ for some }\ \omega_{d,R}>0.\end{align}
On choosing $$X=Q^{\frac{1}{2(R+1)}\omega_{d,R}},$$
  one may obtain a nice upper bound of $\int_{\mathfrak{n}}S_\mathbf{F}(\boldsymbol{\alpha})d\boldsymbol{\alpha}$ provided that both $\kappa_1$ and $\kappa_2$ are large enough depending on $d,R$ and $\omega_{d,R}$. The objective of the next section is to find a nontrivial estimate of $\sup(\mathcal{E})$ as in \eqref{nontrivialE}.
 \vskip3mm

\section{Exponential sums over primes}

Let $m\ge 1$. Suppose that $d_1,\ldots,d_m\ge 1$ and $d_1+\cdots+d_m=d$. Let $\mathcal{B}_m(P)$ be the box in $m$-dimensional space defined by
\begin{align*}b_j'P< x_j\le b_{j}''P,\end{align*}
where $0<b_j'<b_{j}''<1$ are fixed constants for $1\le j\le m$.

We study the exponential summation
\begin{align*}
\Upsilon_m(\alpha)=\sum_{x_1,\ldots,x_m\in \mathcal{B}_m(P)}\Lambda(x_1)\cdots \Lambda(x_m)e\Big(f(x_1,\ldots,x_m)\Big),\end{align*}
where $f$ is a polynomial of $x_1,\ldots,x_m$ of degree $d$ with real coefficients and the coefficient of $x_1^{d_1}\cdots x_m^{d_m}$ in $f$ is $\alpha$. In particular, $f$ is in the form
$$f(x_1,\ldots,x_m)=\alpha x_1^{d_1}\cdots x_m^{d_m}+g(x_1,\ldots,x_m),$$
 where the coefficient of $x_1^{d_1}\cdots x_m^{d_m}$ in $g$ is zero and $\deg(g)\le d$.

We define
\begin{align*}\mathfrak{R}(Q)=\bigcup_{1\le q\le
Q}\bigcup_{\substack{a\in \Z
 \\ (a,q)=1}}\Big[\frac{a}{q}-\frac{Q}{qP^d},\ \frac{a}{q}+\frac{Q}{qP^d}\Big]\end{align*}
and
\begin{align*}\mathfrak{r}(Q)=\R\setminus \mathfrak{R}(Q).\end{align*}
Note that when $R=1$, $\mathfrak{M}(Q)$ and $\mathfrak{m}(Q)$ coincide with $\mathfrak{R}(Q)$ and $\mathfrak{r}(Q)$ in the sense of modulo one.

When $m\ge 2$, one can obtain an upper bound of $\Upsilon_m(\alpha)$ with $\Lambda(\cdot)$ replaced by an arbitrary function.
\begin{lemma}\label{lemma61}Let $m\ge 2$. Suppose that
$Q\le P$. Let $\alpha\in \mathfrak{r}(Q)$. Then we have
\begin{align*}\Upsilon_m(\alpha)\ll &\, P^m(\log P)^{2m}Q^{-\frac{1}{2^d}}.\end{align*}
\end{lemma}
\begin{proof}We define the difference operator $\Delta_1$
$$\Delta_1(\psi(x);x;h)=\psi(x+h)-\psi(x),$$
and $\Delta_j$ iteratively
$$\Delta_j(\psi(x);x;h_1,\ldots,h_j)=\Delta_1(\varphi(x);x;h_j),$$
where $\varphi(x)$ is a function of $x$ depending on $h_1,\ldots,h_{j-1}$ defined as
$$\varphi(x)=\Delta_{j-1}(\psi(x);x;h_1,\ldots,h_{j-1}).$$
Note that if $\psi(x)=x^k$ and $j\le k$, then
$$\Delta_j(\psi(x);x;h_1,\ldots,h_j)=h_1\cdots h_jp(x;j,k,h_1,\ldots,h_j),$$
 where $p(x;j,k,h_1,\ldots,h_j)$ is a polynomial of $x$ of degree $k-j$ with leading coefficient $k!/(k-j)!$. Therefore, if $\psi(x)=x^k$ and $j> k$, then
 $$\Delta_j(\psi(x);x;h_1,\ldots,h_j)=0.$$

We shall first apply the difference argument to deal with the summation over $x_1$. By H\"older's inequality,
\begin{align*}\Upsilon_m(\alpha)^{2^{d_1}}\ll & P^{(m-1)(2^{d_1}-1)}(\log P)^{(m-1)2^{d_1}}
\\ & \times\sum_{x_2,\ldots,x_{m}}\Big|\sum_{x_1}\Lambda(x_1)e\big(f(x_1,\ldots,x_{m})\big)\Big|^{2^{d_1}}.\end{align*}
We apply the difference argument (see Lemma 2.3 in \cite{V}) to deduce that
\begin{align*}&\Big|\sum_{x_1}\Lambda(x_1)e\big(f(x_1,\ldots,x_{m})\big)\Big|^{2^{d_1}}
\\ \ll &\, P^{2^{d_1}-d_1-1}\sum_{\mathbf{h}^{(1)}}\sum_{x_1}\lambda(x_1,\mathbf{h}^{(1)})
e\big(f_1(x_1,\ldots,x_{m},\mathbf{h}^{(1)})\big),\end{align*}
where $\mathbf{h}^{(1)}=(h_1^{(1)},\ldots,h_{d_1}^{(1)})\in \Z^{d_1}$ is a $d_1$-dimensional vector,
\begin{align*}\lambda(x_1,\mathbf{h}^{(1)})=\prod_{I\subseteq \{1,2,\ldots,d_1\}}\Lambda(x_1+\sum_{i\in I}h_i^{(1)})\end{align*}
and
\begin{align*}f_1(x_1,\ldots,x_{m},\mathbf{h}^{(1)})
=\Delta_{d_1}\big(f(x_1,\ldots,x_{m});x_1;\mathbf{h}^{(1)}\big).\end{align*}
Therefore, we obtain
\begin{align*}\Upsilon_m(\alpha)^{2^{d_1}} \ll &\, P^{m(2^{d_1}-1)-d_1}(\log P)^{(m-1)2^{d_1}}
\\ & \times \sum_{x_2,\ldots,x_{m}}\sum_{\mathbf{h}^{(1)}}\sum_{x_1}\lambda(x_1,\mathbf{h}^{(1)})
e\big(f_1(x_1,\ldots,x_{m},\mathbf{h}^{(1)})\big).\end{align*}
Note that $\lambda(x_1,\mathbf{h}^{(1)})\ll (\log P)^{2^{d_1}}$, we further obtain
\begin{align*}\Upsilon_m(\alpha)^{2^{d_1}} \ll &\, P^{m(2^{d_1}-1)-d_1}(\log P)^{m2^{d_1}}
\\ & \times \sum_{\mathbf{h}^{(1)}}\sum_{x_1}\Big|\sum_{x_2,\ldots,x_{m}}
e\big(f_1(x_1,\ldots,x_{m},\mathbf{h}^{(1)})\big)\Big|.\end{align*}

Let $d_j'=d_j$ for $1\le j\le m-1$ and let $d_{m}'=d_m-1$. Then we apply the difference argument $d_j'$ times to the summation over $x_j$ for $2\le j\le m$ and conclude that
\begin{align}\Upsilon_m(\alpha)^{2^{d-1}}\ll &\, P^{m2^{d-1}-d-m+1}(\log P)^{m2^{d-1}}\notag
\\ &\times \sum_{\mathbf{h}^{(1)}}\sum_{x_1}\cdots \sum_{\mathbf{h}^{(m)}}\sum_{x_m}e\big(f_m(x_1,\ldots,x_{m},\mathbf{h}^{(1)},\ldots,\mathbf{h}^{(m)})\big),\label{boundUps}\end{align}
where for $2\le j\le m$, $\mathbf{h}^{(j)}$ is a $d_j'$-dimensional vector in $\Z^{d_j'}$  and $f_j$ is defined iteratively in the following way
\begin{align*}f_j(x_1,\ldots,x_{m},\mathbf{h}^{(1)},\ldots,\mathbf{h}^{(j)})=\Delta_{d_j'}\big(
f_{j-1}(x_1,\ldots,x_{m},\mathbf{h}^{(1)},\ldots,\mathbf{h}^{(j-1)});x_j;\mathbf{h}^{(j)}\big).\end{align*}
Note that $f_m(x_1,\ldots,x_{m},\mathbf{h}^{(1)},\ldots,\mathbf{h}^{(m)})$ can be represented in the form
$$\alpha d_1!\cdots d_m!x_{m} \Pi +\widetilde{f}(x_1,\ldots,x_{m-1},\mathbf{h}^{(1)},\ldots,\mathbf{h}^{(m)}),$$
 where $\widetilde{f}$ is independent of $x_m$ and $\Pi:=\Pi(\mathbf{h}^{(1)},\ldots,\mathbf{h}^{(m)})$ is
\begin{align*}\Pi=\prod_{1\le j\le m}\prod_{1\le i\le d_j'}h_{i}^{(j)}.\end{align*}
Now we conclude from \eqref{boundUps} that
\begin{align}\label{boundups2}\Upsilon_m(\alpha)^{2^{d-1}}\ll P^{m2^{d-1}-d}(\log P)^{m2^{d-1}} \sum_{\mathbf{h}^{(1)}}\cdots \sum_{\mathbf{h}^{(m)}}\min(P,\|\alpha d_1!\cdots d_m!\Pi\|^{-1}).\end{align}

We write
\begin{align*}\mathcal{U}(\alpha)=\sum_{1\le y_1,\ldots,y_{d-1}\le P}\min(P,\|\alpha d_1!\cdots d_m!y_1\cdots y_{d-1}\|^{-1}),\end{align*}
and by \eqref{boundups2},
\begin{align}\label{boundups3}\Upsilon_m(\alpha)^{2^{d-1}}\ll &\, P^{m2^{d-1}-1}(\log P)^{m2^{d-1}} +
P^{m2^{d-1}-d}(\log P)^{m2^{d-1}}\mathcal{U}(\alpha).\end{align}
 By Cauchy's inequality,
\begin{align}\label{boundU1}\mathcal{U}(\alpha)^2\ll P^d(\log P)^{d-1}\sum_{1\le x\le d_1!\cdots d_m!P^{d-1}}\min(P,\|\alpha x\|^{-1}).\end{align}

By Dirichlet's approximation theorem, there exist $a\in \Z, q\in \N$ such that
$$1\le q\le P^{d}Q^{-1}, \ (a,q)=1\ \textrm{ and }\ |\alpha-\frac{a}{q}|\le \frac{Q}{qP^d}.$$
Since $\alpha \in \mathfrak{r}(Q)$, we have $q>Q$. Then by Lemma 2.2 in \cite{V}, we conclude
\begin{align}\label{boundU2}\sum_{1\le x\le d_1!\cdots d_m!P^{d-1}}\min(P,\|\alpha x\|^{-1})\ll P^{d}Q^{-1}\log P,\end{align}
and it follows from \eqref{boundU1} and \eqref{boundU2} that
\begin{align}\label{boundups0}\mathcal{U}(\alpha)\ll P^d(\log P)^{d}Q^{-\frac{1}{2}}.\end{align}
We insert \eqref{boundups0} into \eqref{boundups3} to conclude that
\begin{align*}\Upsilon_m(\alpha)\ll &\, P^m(\log P)^{2m}Q^{-\frac{1}{2^d}}.\end{align*}
This competes the proof.
\end{proof}

When $m=1$, $\Upsilon_m(\alpha)$ is in the form
\begin{align*}\Upsilon(\alpha)=\sum_{b'P<x\le b''P}\Lambda(x)e\big(f(x)\big),\end{align*}
where $0<b'<b''<1$ be two fixed constants and
\begin{align}\label{notationf}f(x):=f(x;\alpha;\alpha_1,\ldots,\alpha_{d-1})=\alpha x^d+ \sum_{j=1}^{d-1}\alpha_jx^j.\end{align}

To deal with $\Upsilon(\alpha)$, we need to make use of the method in the topic on exponential sums over primes. In the special case $f(x)=\alpha x^d$, Kawada and Wooley \cite{KawadaWooley} investigated $\Upsilon(\alpha)$ with applications to the Waring-Goldbach problem. To handle $\Upsilon(\alpha)$ for general $f$, we follow their argument closely (except that Lemma 6.1 of Vaughan \cite{V} does not work when $f$ is a general polynomial). The upper bound of $\Upsilon(\alpha)$ was finally obtained in Lemma \ref{lemmaexpprime}. We point out that Lemma \ref{lemmaexpprime} can be improved with more effects. In particular, when the degree $d\ge 6$, Lemma \ref{lemmaexpprime} can be improved by making use of the recent development \cite{BDG,W2012,W2019} on Vinogradov's mean value theorems.

\begin{lemma}[Weyl's inequality]\label{lemmaweyl}Let $\alpha\in \R$, $a\in \Z$ and $q\in \N$. Suppose that $|\alpha-a/q|\le q^{-2}$ and $(a,q)=1$. Then we have
\begin{align}\label{weyl}\sum_{x\le X}e\Big(\alpha x^d+ \sum_{j=1}^{d-1}\alpha_jx^j\Big)\ll X^{1+\varepsilon} \Big(
\frac{1}{X}+\frac{1}{q}+\frac{q}{X^d}\Big)^{2^{1-d}}.\end{align}
\end{lemma}

 Next we replace $X^{\varepsilon}$ by
$(\log X) q^{\varepsilon}$ in \eqref{weyl} when $q$ is small.
\begin{lemma}\label{lemmaweyl2}Let $\alpha\in \R$, $a\in \Z$ and $q\in \N$. Assume $X\ge 2$. Suppose that $|\alpha-a/q|\le \frac{1}{2qX^{d-1}}$, $(a,q)=1$ and $q\le X$. Then we have
\begin{align}\label{weylmajor}\sum_{1\le x\le X}e\Big(\alpha x^d+ \sum_{j=1}^{d-1}\alpha_jx^j\Big) \ll \frac{X(\log X)q^\varepsilon}{ \big(q+X^d|q\alpha-a|\big)^{2^{1-d}}}.\end{align}
\end{lemma}
\begin{proof}By the difference  argument (see Lemma 2.3 in \cite{V}), we have
\begin{align}\label{boundtheta0}\Big|\sum_{1\le x\le X}e\big(\alpha x^d+ \sum_{j=1}^{d-1}\alpha_jx^j\big)\Big|^{2^{d-1}}\ll X^{2^{d-1}-1}+X^{2^{d-1}-d}\Theta,\end{align}
where
\begin{align*}\Theta=\sum_{1\le h_1,\ldots,h_{d-1}\le X}
\min(X,\ \|h_1\cdots h_{d-1}\alpha\|^{-1} ).\end{align*}
Note that
$$|h_1\cdots h_{d-2}\alpha-h_1\cdots h_{d-2}\frac{a}{q}|\le \frac{1}{2qX}.$$
By Lemma 2.2 in \cite{V}, we have
\begin{align*}\sum_{1\le h_{d-1}\le X}
\min(X,\ \|h_1\cdots h_{d-1}\alpha\|^{-1} )\ll X^2(\log X)q^{-1}(q,h_1\cdots h_{d-2}),\end{align*}
and therefore,
\begin{align*}\Theta\ll X^2(\log X)q^{-1}\sum_{\substack{1\le h_1,\ldots,h_{d-2}\le X}}(q,h_1\cdots h_{d-2})
.\end{align*}
On applying the elementary inequality
$$\sum_{h\le X}(q,h)\ll Xq^\varepsilon,$$
one can further obtain
\begin{align}\label{thetafirstbound}\Theta\ll X^d(\log X)q^{-1+\varepsilon}
.\end{align}

Note that
\begin{align}\label{over2q}|h_1\cdots h_{d-1}\alpha-h_1\cdots h_{d-1}\frac{a}{q}|\le \frac{1}{2q},\end{align}
and we conclude
\begin{align}\label{theta12}\Theta\le \Theta_1+\Theta_2,\end{align}
where
\begin{align*}\Theta_1= \sum_{\substack{1\le h_1,\ldots,h_{d-1}\le X
\\ q|h_1\cdots h_{d-1}}}\min(X,\ \|h_1\cdots h_{d-1}(\alpha-\frac{a}{q})\|^{-1} )\end{align*}
and \begin{align*}\Theta_2=
\sum_{\substack{1\le h_1,\ldots,h_{d-1}\le X
\\ q\nmid h_1\cdots h_{d-1}}}\min(2q,\ \|h_1\cdots h_{d-1}\alpha\|^{-1} ).\end{align*}

By \eqref{over2q}, one has\begin{align}\label{boundtheta1}\Theta_1 \le \sum_{\substack{h_1,\ldots,h_{d-1}
\\ q|h_1\cdots h_{d-1}}}\frac{1}{h_1\cdots h_{d-1}|\alpha-\frac{a}{q}|} \ll \frac{(\log X)^{d-1}q^\varepsilon}{|q\alpha-a|}.\end{align}
To obtain the upper bound for $\Theta_2$, we conclude from Lemma 2.1 in \cite{V} that
\begin{align*}
\sum_{\substack{1\le h_{d-1}\le X}}\min(2q,\ \|h_1\cdots h_{d-1}\alpha\|^{-1} )\ll X(\log X)(q,h_1\cdots h_{d-2}),\end{align*}
and then it follows that
\begin{align}\label{boundtheta2}\Theta_2\ll
\sum_{\substack{1\le h_1,\ldots,h_{d-2}\le X}}X(\log X) (q,h_1\cdots h_{d-2})\ll X^{d-1}(\log X)q^\varepsilon. \end{align}
Inserting \eqref{boundtheta1} and \eqref{boundtheta2} into \eqref{theta12}, we obtain
\begin{align}\label{thetasecondbound}\Theta\ll \frac{(\log X)^{d-1}q^\varepsilon}{|q\alpha-a|}+X^{d-1}(\log X)q^{\varepsilon}\ll \frac{(\log X)^{d-1}q^\varepsilon}{|q\alpha-a|}
.\end{align}

We conclude from \eqref{thetafirstbound} and \eqref{thetasecondbound} that
\begin{align}\label{thetabound}\Theta\ll \frac{X^{d}(\log X)^{d-1}q^\varepsilon}{q+X^d|q\alpha-a|}
.\end{align}
Now \eqref{weylmajor} follows from \eqref{boundtheta0} and \eqref{thetabound}.
\end{proof}

\begin{lemma}\label{lemmaweyl3}Let $\alpha\in \R$, $a\in \Z$ and $q\in \N$. Assume $X\ge 2$. Suppose that $|\alpha-a/q|\le \frac{1}{2qX^{d-1}}$, $(a,q)=1$ and $q\le 2X^{d-1}$. Then we have
\begin{align}\label{weylall}\sum_{1\le x\le X}e\Big(\alpha x^d+ \sum_{j=1}^{d-1}\alpha_jx^j\Big) \ll \frac{X(\log X)q^\varepsilon}{ \big(q+X^d|q\alpha-a|\big)^{2^{1-d}}}+X^{1-2^{1-d}+\varepsilon}.\end{align}
\end{lemma}
\begin{proof}When $q\ge X$, \eqref{weylall} follows from Lemma \ref{lemmaweyl}. When $q<X$, \eqref{weylall}  follows from Lemma \ref{lemmaweyl2}.\end{proof}

\begin{lemma}\label{lemmadivisor}Let $d\ge 2$ be a natural number. Let $0<\Delta< \frac{1}{d-1}$. Then one has
\begin{align}\label{boundy12}\sum_{1\le y_1<y_2\le Y}\frac{(y_2^d-y_1^d,q)^{\Delta}}{(y_2-y_1)^{\Delta}}\ll Y^{2-\Delta} q^{\varepsilon}
.\end{align}\end{lemma}
\begin{proof}
On introducing $x=(y_1,y_2)$ and changing variables, we obtain
\begin{align}\label{boundy12-1}\sum_{1\le y_1<y_2\le Y}\frac{(y_2^d-y_1^d,q)^{\Delta}}{(y_2-y_1)^{\Delta}}\le \sum_{x\le Y}\frac{(x^d,q)^{\Delta}}{x^{\Delta}}\sum_{\substack{1\le z_1<z_2\le Y/x
\\ (z_1,z_2)=1}}\frac{(z_2^d-z_1^d,q)^{\Delta}}{(z_2-z_1)^{\Delta}}.\end{align}
Write $t=z_2-z_1$ and $w_t(z)=\frac{(z+t)^d-z^d}{t}$. Then we have
\begin{align*}\sum_{\substack{1\le z_1<z_2\le Y/x
\\ (z_1,z_2)=1}}\frac{(z_2^d-z_1^d,q)^{\Delta}}{(z_2-z_1)^{\Delta}}\le \sum_{t\le Y/x}\frac{(t,q)^{\Delta}}{t^{\Delta}}\sum_{\substack{1\le z\le Y/x
\\ (z,t)=1}}(w_t(z),q)^{\Delta}.\end{align*}
On introducing $s=(w_t(z),q)$, we further obtian
\begin{align}\label{outerinner}\sum_{\substack{1\le z_1<z_2\le Y/x
\\ (z_1,z_2)=1}}\frac{(z_2^d-z_1^d,q)^{\Delta}}{(z_2-z_1)^{\Delta}}\le
\sum_{\substack{t\le Y/x}}\frac{(t,q)^{\Delta}}{t^{\Delta}}\sum_{\substack{s|q
\\ s\le d(Y/x)^{d-1}}}s^{\Delta}
\sum_{\substack{1\le z\le Y/x
\\ (z,t)=1\\ s|w_t(z)}}1.\end{align}

One has
\begin{align*}
\sum_{\substack{1\le z\le Y/x
\\ (z,t)=1\\ s|w_t(z)}}1\ll \Big(\frac{Y}{xs}+1\Big)s^\varepsilon,\end{align*}
and it follows that
\begin{align}\label{innersy}\sum_{\substack{s|q
\\ s\le d(Y/x)^{d-1}}}s^{\Delta}
\sum_{\substack{1\le z\le Y/x
\\ (z,t)=1\\ s|w_t(z)}}1 \ll \frac{Y}{x}q^\varepsilon+\big(\frac{Y}{x}\big)^{(d-1)\Delta}q^{\varepsilon}
\ll \frac{Y}{x}q^\varepsilon.\end{align}

Inserting \eqref{innersy} into \eqref{outerinner}, we obtain
\begin{align*}\sum_{\substack{1\le z_1<z_2\le Y/x
\\ (z_1,z_2)=1}}\frac{(z_2^d-z_1^d,q)^{\Delta}}{(z_2-z_1)^{\Delta}}\ll \big(\frac{Y}{x}\big)^{2-\Delta}q^\varepsilon.\end{align*}
and by \eqref{boundy12-1},
\begin{align}\label{boundy12-2}\sum_{1\le y_1<y_2\le Y}\frac{(y_2^d-y_1^d,q)^{\Delta}}{(y_2-y_1)^{\Delta}}\ll
Y^{2-\Delta}q^\varepsilon\sum_{x\le Y}\frac{(x^d,q)^{\Delta}}{x^{2}}.\end{align}

On noting that $(x^d,q)^\Delta\le (x,q)^{d\Delta}\le (x,q)x^{\Delta}$, we deduce that
\begin{align}\label{outerx}\sum_{x\le Y}\frac{(x^d,q)^{\Delta}}{x^{2}}\ll q^\varepsilon.\end{align}
Now \eqref{boundy12} follows from \eqref{boundy12-2} and \eqref{outerx}.\end{proof}

Let $\tau(\cdot)$ denote the divisor function. By H\"older's inequality, we have
\begin{align*}&\sum_{1\le y_1<y_2\le Y}\tau(y_1)\tau(y_2)\frac{(y_2^d-y_2^d,q)^{2^{1-d}}}{(y_2-y_1)^{2^{1-d}}}
\\  &\ \  \ll   \Big(Y^2(\log Y)^{14}\Big)^{1/3}\Big(\sum_{1\le y_1<y_2\le Y}\frac{(y_2^d-y_2^d,q)^{3\cdot 2^{-d}}}{(y_2-y_1)^{3\cdot 2^{-d}}}\Big)^{2/3}.\end{align*}
Note that $3\cdot 2^{-d}< \frac{1}{d-1}$. We deduce from Lemma \ref{lemmadivisor} that
\begin{align}\label{bounddivisor1}\sum_{1\le y_1<y_2\le Y}\tau(y_1)\tau(y_2)\frac{(y_2^d-y_2^d,q)^{2^{1-d}}}{(y_2-y_1)^{2^{1-d}}}
 \ll   Y^{2-2^{1-d}}(\log Y)^{5}
q^{\varepsilon}.\end{align}

Suppose that $\xi(x)$ and $\eta(y)$ are two sequences of real numbers satisfying
$$|\xi(x)|\le \tau(x) \ \textrm{ and }\ |\eta(y)|\le \tau(y).$$
 We freely use the elementary inequality
\begin{align*}\sum_{ x\le X}\tau(x)^{r}\ll X (\log X)^{2^r-1},\end{align*}
where $r\in \Z^{+}$ is a fixed positive integer.
Recall $f(x)$ in \eqref{notationf} and we introduce
\begin{align*}\Psi(\alpha)=\sum_{x\sim X}\xi(x)\sum_{\substack{y\sim Y \\ y\le P/x}}\eta(y) e\big(f(xy)\big),\end{align*}
where $w\sim W$ means $\frac{W}{2}<w\le W$.
\begin{lemma}\label{lemmatype2}Suppose that $Y\le X$, $P\ll XY\ll P$ and $B\le X^{1/2}$. Let $\alpha\in \mathfrak{r}(B)$. Then we have
\begin{align*}\Psi(\alpha)\ll XY(\log X)^{5} B^{-2^{-d}+\varepsilon}+XY^{1/2}(\log X)^{3}.\end{align*}
\end{lemma}
\begin{proof}By Cauchy's inequality, one has
\begin{align*}\Psi(\alpha)^2 \ll X(\log X)^{3}\sum_{x\sim X}\Big|\sum_{\substack{y\sim Y \\ y\le P/x}}\eta(y) e\big(f(xy)\big)\Big|^2.\end{align*}
On expanding the square and interchanging the order of summations, we obtain
\begin{align*}\Psi(\alpha)^2 \ll X(\log X)^{3}
\sum_{\substack{y_1\sim Y }}\sum_{\substack{y_2\sim Y }}\eta(y_1)\eta(y_2)\sum_{x\in I}e\big(f(xy_2)-f(xy_1)\big),\end{align*}
where $I\subseteq [X/2, X]$ denotes an interval depending on $y_1$ and $y_2$. The contribution from $y_1=y_2$ is at most $O(X^2Y\log^6 X)$, and by symmetry we conclude that
\begin{align}\label{boundpsisquare}\Psi(\alpha)^2 \ll X^2Y(\log X)^{6}+X(\log X)^3\widetilde{\Psi},\end{align}
where\begin{align}\label{tildepsi}\widetilde{\Psi} =
\sum_{\substack{Y/2\le y_1<y_2\le Y }}\tau(y_1)\tau(y_2)\big|\Xi(\alpha)\big|\end{align}
and
\begin{align*}\Xi(\alpha):=\Xi(\alpha;y_1,y_2) =\sum_{x\in I}e\big(f(xy_2)-f(xy_1)\big).\end{align*}

Let
\begin{align}\label{definetandT}t=y_2-y_1 \ \textrm{ and } \ T=(y_2^d-y_1^d)(y_2-y_1)^{-1}.\end{align}
As a polynomial of $x$, the leading term in $f(xy_2)-f(xy_1)$ is
$\alpha(y_2^d-y_1^d)x^d=\alpha tTx^d$. Note that $1\le t\le Y$ and $d(Y/2)^{d-1}<T\le dY^{d-1}$.

By Dirichlet's approximation theorem, there exist $b\in \Z$ and $r\in \N$ such that
\begin{align*}(b,r)=1, \ r\le 2X^{d-1} \ \textrm{ and } \  |tT\alpha-\frac{b}{r}|\le \frac{1}{2rX^{d-1}}.\end{align*}
On applying Lemma \ref{lemmaweyl3}, we have
\begin{align}\label{innertype2}\Xi(\alpha)\ll X(\log X) r^{\varepsilon}\Big(r+X^d|rtT\alpha-b|\Big)^{-2^{1-d}}+X^{1-2^{1-d}+\varepsilon}.\end{align}
If $r>B$, then
\begin{align}\label{boundsumxI}\Xi(\alpha)\ll X(\log X)B^{-2^{1-d}+\varepsilon}.\end{align}
Then we obtain from \eqref{tildepsi} and \eqref{boundsumxI}
\begin{align}\label{boundpsi2}\widetilde{\Psi} \ll XY^2(\log X)^{6} B^{-2^{1-d}+\varepsilon}.\end{align}
Similarly, if $r\le B/3$ and $|tT\alpha-b/r|>\frac{B}{3rX^d}$, then \eqref{boundsumxI} and \eqref{boundpsi2} hold as well.
Therefore, we next assume
\begin{align}\label{brnew}r\le \frac{B}{3} \ \textrm{ and } \  |tT\alpha-\frac{b}{r}|\le \frac{B}{3rX^{d}}.\end{align}

Let $W=X^{d/2}Y^{1/2}$. By Dirichlet's approximation theorem, there exist $c\in \Z$ and $s\in \N$ such that
\begin{align}\label{existcs} (c,s)=1, \ s\le W \ \textrm{ and } \  |T\alpha-\frac{c}{s}|\le \frac{1}{sW}.\end{align}
We deduce from \eqref{brnew} and \eqref{existcs} that
\begin{align*}
\big|\frac{b}{r}-t\frac{c}{s}\big|\le &\big|tT\alpha-\frac{b}{r}\big|+\big|tT\alpha-t\frac{c}{s}\big|
\\ \le&
\frac{B}{3rX^{d}} +\frac{t}{sW}=\frac{1}{3rs}\Big(\frac{Bs}{X^{d}} +\frac{3tr}{W}\Big)
\\ \le &\frac{1}{3rs}\Big(\frac{BW}{X^{d}} +\frac{YB}{W}\Big)=
\frac{2BY^{1/2}}{3rsX^{d/2}}.\end{align*}
Note that for $d\ge 2$, $Y\le X$ and $B\le X^{1/2}$, we have $\frac{2BY^{1/2}}{3X^{d/2}}<1$. Therefore, we obtain
$\big|\frac{b}{r}-t\frac{c}{s}\big|<\frac{1}{rs}$, and this gives
\begin{align}\label{linkbrcs}\frac{b}{r}=t\frac{c}{s} \ \textrm{ and } \  r=\frac{s}{(t,s)}.\end{align}

It follows from \eqref{innertype2} and \eqref{linkbrcs} that
\begin{align*}\Xi(\alpha)\ll X(\log X) s^{\varepsilon}(t,s)^{2^{1-d}-\varepsilon}\Big(s+X^dt|sT\alpha-c|\Big)^{-2^{1-d}}+X^{1-2^{1-d}+\varepsilon}.\end{align*}
By \eqref{tildepsi}, we further obtain
\begin{align*}\widetilde{\Psi}\ll XY^2(\log X)^4 s^{\varepsilon}\Big(s+X^dt|sT\alpha-c|\Big)^{-2^{1-d}}+X^{1-2^{1-d}+\varepsilon}Y^2.\end{align*}
Similarly to the argument leading to \eqref{brnew}, one can obtain \eqref{boundpsi2}
unless
\begin{align}\label{csnew}s\le \frac{B}{3d} \ \textrm{ and } \  |T\alpha-\frac{c}{s}|\le \frac{B}{3sX^dt}.\end{align}

 From now on, we assume that \eqref{csnew} holds. Let $V=X^{d/2}Y^{(d-1)/2}$. By Dirichlet's approximation theorem, there exist
 $a\in \Z$ and $q\in \N$ such that
\begin{align}\label{existaq} (a,q)=1, \ q\le V \ \textrm{ and } \  |\alpha-\frac{a}{q}|\le \frac{1}{qV}.\end{align}
We deduce from \eqref{csnew} and \eqref{existaq} that
\begin{align*}|\frac{c}{s}-\frac{aT}{q}|\le & |T\alpha-\frac{c}{s}|+|T\alpha-\frac{aT}{q}|
\\ \le & \frac{B}{3sX^dt}+\frac{T}{qV}=\frac{1}{3sq}\Big(\frac{Bq}{X^dt}+\frac{3Ts}{V}\Big)
\\ \le & \frac{1}{3sq}\Big(\frac{BV}{X^d}+\frac{Y^{d-1}B}{V}\Big)= \frac{2BY^{(d-1)/2}}{3sqX^{d/2}}. \end{align*}
Since $Y\le X$ and $B\le X^{1/2}$, one has $BY^{(d-1)/2}< X^{d/2}$ and
$|\frac{c}{s}-\frac{aT}{q}|<\frac{1}{sq}$. In particular, we have $\frac{c}{s}=\frac{aT}{q}$. Then by \eqref{definetandT} and \eqref{linkbrcs}, we obtain
\begin{align}\label{linkbraq}\frac{b}{r}=tT\frac{a}{q} \ \textrm{ and } \  r=\frac{q}{(tT,q)}=\frac{q}{(y_2^d-y_1^d,q)}.\end{align}
We conclude from \eqref{innertype2} and \eqref{linkbraq} that
\begin{align*}\Xi(\alpha)\ll X(\log X) q^{\varepsilon}(y_2^d-y_1^d,q)^{2^{1-d}}
\Big(q+X^dtT|q\alpha-a|\Big)^{-2^{1-d}}+X^{1-2^{1-d}+\varepsilon}.\end{align*}
Note that
\begin{align*}
\Big(q+X^dtT|q\alpha-a|\Big)^{-1}\ll (y_2-y_1)^{-1}Y\Big(q+X^dY^d|q\alpha-a|\Big)^{-1}.\end{align*}
We deduce that
\begin{align*}\Xi(\alpha)\ll \frac{XY^{2^{1-d}}(\log X) q^{\varepsilon}}{\big(q+X^dY^d|q\alpha-a|\big)^{2^{1-d}}}\cdot
\frac{(y_2^d-y_1^d,q)^{2^{1-d}}}{(y_2-y_2)^{2^{1-d}}}+X^{1-2^{1-d}+\varepsilon}.\end{align*}

Now we conclude  from \eqref{bounddivisor1} and \eqref{tildepsi} that
\begin{align*}\widetilde{\Psi}  \ll
\frac{XY^2(\log X)^{6} q^{\varepsilon}}{(q+X^dY^d|q\alpha-a|)^{2^{1-d}}}+X^{1-2^{1-d}+\varepsilon}Y^2.\end{align*}
Then we established \eqref{boundpsi2} due to $\alpha\in \mathfrak{r}(B)$.  We complete the proof by inserting \eqref{boundpsi2} into \eqref{boundpsisquare}.
\end{proof}

Lemma \ref{lemmatype2} provides the type II estimate and the next result provides the type I estimate.
\begin{lemma}\label{lemmatype1}Suppose that $B\le \frac{P}{X}$ and $B\le \frac{P^{d/2}}{X^d}$. Let $\alpha \in \mathfrak{r}(B)$. Then we have
\begin{align}\label{boundtype1}\sum_{x\le X}\xi(x)\sum_{y\le P/x}e\big(f(xy)\big)\ll P(\log B)^{3}B^{-2^{1-d}+\varepsilon},\end{align}
and
\begin{align}\label{boundtype1log}\sum_{x\le X}\xi(x)\sum_{y\le P/x}(\log y)e\big(f(xy)\big)\ll P(\log B)^{4}B^{-2^{1-d}+\varepsilon}.\end{align}
\end{lemma}
\begin{proof}Note that \eqref{boundtype1log} follows from \eqref{boundtype1} and the partial summation formula
\begin{align*}\sum_{y\le P/x}(\log y)e\big(f(xy)\big)=\sum_{y\le P/x}e\big(f(xy)\big)\log (P/x)-
\int_{1}^{P/x}\sum_{y\le \gamma}e\big(f(xy)\big)\frac{d\gamma}{\gamma}.\end{align*}
Therefore, we only need to establish  \eqref{boundtype1}.

By Dirichlet's approximation theorem, there exist $b\in \Z$ and $r\in \N$ such that
\begin{align*}(b,r)=1,\ r\le 2(P/x)^{d-1}\ \textrm{ and } \ |x^d\alpha-\frac{b}{r}|\le \frac{1}{2r(P/x)^{d-1}}.\end{align*}
On applying Lemma \ref{lemmaweyl3}, we have
\begin{align}\label{boundtype1inner}\sum_{y\le P/x}e\big(f(xy)\big)\ll  \frac{(\log P)(P/x)r^{\varepsilon}}{\big(r+(P/x)^d|rx^d\alpha-b|\big)
^{2^{1-d}}}+(P/x)^{1-2^{1-d}+\varepsilon}.\end{align}
Note that $B\le P/X\le P/x$ for $x\le X$. We conclude from \eqref{boundtype1inner} that
\begin{align*}\sum_{y\le P/x}e\big(f(xy)\big)\ll & (\log P)(P/x)B^{-2^{1-d}+\varepsilon}\end{align*}
unless
\begin{align}\label{type1br}r\le \frac{B}{3}\ \textrm{ and } \ |x^d\alpha-\frac{b}{r}|\le \frac{B}{3r(P/x)^{d}}.\end{align}

From now on, we assume that \eqref{type1br} holds. By Dirichlet's approximation theorem, there exist $a\in \Z$ and $q\in \N$ such that
\begin{align}\label{type1aq}(a,q)=1,\ q\le P^{d/2}\ \textrm{ and } \ |\alpha-\frac{a}{q}|\le \frac{1}{qP^{d/2}}.\end{align}
Then we deduce by \eqref{type1br} and \eqref{type1aq} that
\begin{align*}\big|\frac{b}{r}-\frac{ax^d}{q}\big|\le &
\big|x^d\alpha-\frac{b}{r}\big|+\big|x^d\alpha-\frac{ax^d}{q}\big|
\\ \le & \frac{Bx^d}{3rP^{d}}+\frac{x^d}{qP^{d/2}}
\le \frac{1}{3rq}\Big(\frac{qBx^d}{P^d}+\frac{3rx^d}{P^{d/2}}\Big)\le \frac{2BX^d}{3rqP^{d/2}}<\frac{1}{rq},\end{align*}
and it follows that
\begin{align}\label{type1linkbraq}\frac{b}{r}=\frac{ax^d}{q} \ \textrm{ and } \ r=\frac{q}{(q,x^d)}.\end{align}

We obtain from \eqref{boundtype1inner} and \eqref{type1linkbraq} that
\begin{align*}\sum_{y\le P/x}e\big(f(xy)\big)\ll (\log P)(P/x)\frac{q^{\varepsilon}(q,x^d)^{2^{1-d}-\varepsilon}}{\big(q+P^d|q\alpha-a|\big)
^{2^{1-d}}}+(P/x)^{1-2^{1-d}+\varepsilon}.\end{align*}
Then we deduce that
\begin{align*}\sum_{x\le X}\xi(x)\sum_{y\le P/x}e\big(f(xy)\big)
 \ll \frac{P(\log P)^{3}q^{\varepsilon}}{(q+P^d|q\alpha-a|)
^{2^{1-d}}}+P(P/X)^{-2^{1-d}+\varepsilon}.\end{align*}
This completes the proof since $\alpha \in \mathfrak{r}(B)$ and $B\le P/X$.
\end{proof}

\begin{lemma}\label{lemmaexpprime}Let $0<b'<b''<1$ be two fixed constants. Suppose that $Q\le P^{1/4}$. Let $\alpha\in \mathfrak{r}(Q)$. Then we have
\begin{align*}\sum_{b'P<x\le b''P}\Lambda(x)e\big(f(x)\big)\ll P(\log P)^{7}Q^{-\frac{1}{2^{d}}+\varepsilon}.\end{align*}
\end{lemma}
\begin{proof}
Let
$$W=P^{1/8}.$$
We apply Vaughan's identity to conclude that for $b'P< n\le b''P$,
\begin{align*}
\Lambda(n)=\sum_{\substack{xy=n\\ 1\le x\le W}}\mu(x)\log y -
\sum_{\substack{xy=n\\ 1\le x\le W^2}}\xi(x)-\sum_{\substack{xy=n\\ x>W \\ y>W}}\eta(x)\Lambda(y),  \end{align*}
where
$$\xi(x)=\sum_{\substack{1\le x_1,x_2\le W^2 \\ x_1x_2=x}}\mu(x_1)\Lambda(x_2),\ \ \
\eta(x)=\sum_{\substack{1\le x_1\le W \\ x_1x_2=x}}\mu(x_1),$$

Then we deduce that
\begin{align*}\sum_{b'P<x\le b''P}\Lambda(x)e\big(f(x)\big)=\Psi_1+\Psi_1'+\Psi_2,\end{align*}
where
\begin{align*}\Psi_1=\sum_{x\le W}\mu(x) \sum_{b'P/x< y\le b''P/x}(\log y)e\big(f(xy)\big),\end{align*}
\begin{align*}\Psi_1'=\sum_{x\le W^2}\xi(x) \sum_{b'P/x< y\le b''P/x}e\big(f(xy)\big)\end{align*}
and
\begin{align*}\Psi_2=\sum_{x> W}\eta(x) \sum_{\substack{b'P/x< y\le b''P/x
\\ y>W}}\Lambda(y)e\big(f(xy)\big).\end{align*}
We deduce from Lemma \ref{lemmatype1} that
\begin{align*}|\Psi_1|+|\Psi_1'|\ll P(\log P)^{4}Q^{-\frac{1}{2^{d-1}}+\varepsilon}.\end{align*}
By Lemma \ref{lemmatype2} and the dyadic argument,
\begin{align*}\Psi_2\ll P(\log P)^{7}Q^{-\frac{1}{2^{d}}+\varepsilon}.\end{align*}
This completes the proof.
\end{proof}

For $h_1,\ldots,h_R\in \Z[x_1,\ldots,x_t]$, we use
$$\rank(\mathbf{h})=\rank(h_1,\ldots,h_R)$$
to denote the rank of $\{h_1,\ldots,h_R\}$ by viewing $h_1,\ldots,h_R$ as vectors in the linear space $\R[x_1,\ldots,x_t]$ over $\R$. In other words, $\rank(\mathbf{h})$ is the dimension of the subspace of $\R[x_1,\ldots,x_t]$ generated by $h_1,\ldots,h_R$.

For $1\le i\le R$, let $h_i(x_1,\ldots,x_t)$ be a form of degree $d$ and let  $g_i(x_1,\ldots,x_t)$ be a polynomial of $\mathbf{x}=(x_1,\ldots,x_t)$ of degree smaller than $d$.
Let
$$f_i=h_i+g_i\ \textrm{ for } \ 1\le i\le R.$$
We define
\begin{align}\label{defineE}\mathcal{E}(\boldsymbol{\alpha})=\sum_{x_1,\ldots,x_t\in \mathcal{B}_t(P)}\Lambda(x_1)\cdots \Lambda(x_t)e\Big(\boldsymbol{\alpha}\cdot \mathbf{f}(x_1,\ldots,x_t)\Big).\end{align}

\begin{lemma}\label{lemmaboundE1}Let $\mathcal{E}(\boldsymbol{\alpha})$ be given in \eqref{defineE}. Suppose that $\rank(\mathbf{h})=R$. Let $Q\le P^{1/4}$ and $\boldsymbol{\alpha}\in \mathfrak{m}(Q)$. Then we have
\begin{align}\label{boundE}\mathcal{E}(\boldsymbol{\alpha})\ll P^{t}(\log P)^{7t}Q^{-\frac{1}{2^{d}R}+\varepsilon}.\end{align}
\end{lemma}
\begin{proof}We represent $h_i(x_1,\ldots,x_t)$ in the form
\begin{align}\label{coeff2}h_i(x_1,\ldots,x_t)=\sum_{1\le j_0\le \ldots \le j_{d-1}\le t}h^{(i)}_{j_0,\ldots,j_{d-1}}x_{j_0}\cdots x_{j_{d-1}},\end{align}
where the coefficients $h_{j_0,j_1,\ldots,j_{d-1}}^{(i)}$ are integers. Note that the notation of the coefficient in \eqref{coeff2} is different from \eqref{coeff1}.
For each $1\le k\le R$, we choose a vector
\begin{align}\label{mathbfj}\mathbf{j}_k=(j_0^{(k)},\ldots,j_{d-1}^{(k)})\in \N^d\end{align}
with $1\le j_0^{{k}}\le \ldots\le j_{d-1}^{(k)}\le t$, and we introduce the square matrix
\begin{align*}B(\mathbf{j}_1,\ldots,\mathbf{j}_R)=\big(h_{\mathbf{j}_k}^{(i)}\big)_{1\le i,k\le R}.\end{align*}

Since $\rank(\mathbf{h})=R$, we can choose $\mathbf{j}_1,\ldots,\mathbf{j}_R$ such that $B:=B(\mathbf{j}_1,\ldots,\mathbf{j}_R)$ is invertible.
Let
$$\boldsymbol{\beta}=\boldsymbol{\alpha}B.$$
We view $\boldsymbol{\alpha}\cdot \mathbf{f}(x_1,\ldots,x_t)$ as a polynomial of $x_1,\ldots,x_t$, and the coefficient of $x_{j_0^{(k)}}\cdots x_{j_{d-1}^{(k)}}$ is the $k$-th coordinate of $\boldsymbol{\beta}$. Let
$C_{\mathbf{h}}$ denote the maximum of the absolute values of coefficients of $\mathbf{h}$. Let
\begin{align}\label{chooseY}Y=\frac{1}{RC_{\mathbf{h}}}Q^{\frac{1}{R}}.\end{align}

We claim that $\beta_k\in \mathfrak{r}(Y)$ for some $1\le k\le R$. Otherwise, we have $\beta_k\in \mathfrak{R}(Y)$ for all $1\le k\le R$. In particular, there exist $b_k\in \Z$, $r_k\in \N$ and $\gamma_k\in \R$ such that
\begin{align}\label{brk}1\le r_k\le Y,\ (b_k,r_k)=1,\ |\gamma_k|\le \frac{Y}{r_kP^{d}} \ \textrm{ and }\ \beta_k=\frac{b_k}{r_k}+\gamma_k.\end{align}
Let $B_{j,k}$ denote the $(j,k)$-th algebraic minor of $B$. We introduce
\begin{align}\label{aqR}a_j'=r_1\cdots r_R\sum_{k=1}^R\frac{b_k}{r_k}B_{j,k},\, \ q'=r_1\cdots r_R\det(B) \ \textrm{ and }\
\eta_j=\frac{1}{\det(B)}\sum_{k=1}^R\gamma_kB_{j,k}.\end{align}
Note that $|B_{j,k}|\le (R-1)!C_{\mathbf{h}}^{R-1}$ and $|\det(B)|\le R!C_{\mathbf{h}}^{R}$.
We deduce that
$$|q'|\le R!C_{\mathbf{h}}^{R}Y^{R} \ \textrm{ and }\ |\eta_j|\le \frac{(R-1)!C_{\mathbf{h}}^{R-1}Y}{|\det(B)|P^d} \sum_{k=1}^R\frac{1}{r_k}\le R!C_{\mathbf{h}}^{R-1} \frac{Y^R}{|q'|P^d}.$$
By \eqref{chooseY}, we have $|q'|\le Q$ and $|\eta_j|\le \frac{Q}{|q'|P^d}$.

We deduce from $\boldsymbol{\beta}=\boldsymbol{\alpha}B$ that
\begin{align}\label{alphajeq}\alpha_j=\frac{1}{\det(B)}\sum_{k=1}^R\beta_kB_{j,k}.\end{align}
Then we conclude from \eqref{brk}, \eqref{aqR} and \eqref{alphajeq} that
$$\alpha_j=\frac{a_j'}{q'}+\eta_j,\ j=1,\ldots,R.$$
This is a contradiction to $\boldsymbol{\alpha}\in \mathfrak{m}(Q)$. Therefore, we established the claim that $\beta_k\in \mathfrak{r}(Y)$ for some $1\le k\le R$.

We represent $x_{j_0^{(k)}}\cdots x_{j_{d-1}^{(k)}}$ in the form $x_{i_1}^{d_1}\cdots x_{i_m}^{d_m}$
with $i_1,\ldots,i_m$ pairwise distinct.
Then $\boldsymbol{\alpha}\cdot \mathbf{f}(x_1,\ldots,x_t)$ can be represented in the form
\begin{align*}\boldsymbol{\alpha}\cdot \mathbf{f}(x_1,\ldots,x_t)=\beta_k x_{i_1}^{d_1}\cdots x_{i_m}^{d_m}+g(x_1,\ldots,x_t),\end{align*}
where the coefficient of $x_{i_1}^{d_1}\cdots x_{i_m}^{d_m}$ in $g(x_1,\ldots,x_t)$ is zero and $\deg(g)\le d$.
 Now \eqref{boundE} follows from Lemma \ref{lemma61} in the case $m\ge 2$ and from Lemma \ref{lemmaexpprime} in the case $m=1$.
\end{proof}

\begin{lemma}\label{lemmaboundE2}Let $\mathcal{E}(\boldsymbol{\alpha})$ be given in \eqref{defineE}. Suppose that $\rank(\mathbf{h})=R$. Let $Q\le \frac{1}{4}P^{\frac{Rd}{R+1}}$ and $\boldsymbol{\alpha}\in \mathfrak{m}(Q)$. Then we have
\begin{align}\label{boundE2}\mathcal{E}(\boldsymbol{\alpha})\ll P^{t}(\log P)^{7t}Q^{-\frac{R+1}{2^{d+2}dR^2}+\varepsilon}.\end{align}
\end{lemma}
\begin{proof}By Lemma \ref{lemmaboundE1}, we only need to prove \eqref{boundE2} when $P^{1/4}\le Q\le \frac{1}{4}P^{\frac{Rd}{R+1}}$. For $Q\ge P^{1/4}$,  Lemma \ref{lemmaboundE1} implies
\begin{align}\label{boundE3}\mathcal{E}(\boldsymbol{\alpha})\ll P^{t}(\log P)^{7t}P^{-\frac{1}{2^{d+2}R}+\varepsilon}.\end{align}
For $Q\le \frac{1}{4}P^{\frac{Rd}{R+1}}$, one has
\begin{align}\label{linkQP}Q^{\frac{R+1}{2^{d+2}dR^2}}\le P^{\frac{1}{2^{d+2}R}}.\end{align}
Now \eqref{boundE2} follows from \eqref{boundE3} and \eqref{linkQP} when $P^{1/4}\le Q\le \frac{1}{4}P^{\frac{Rd}{R+1}}$.
 \end{proof}

\vskip3mm

\section{The minor arcs estimate}

Suppose that $\mathbf{F}=(F_1,\ldots,F_R)$  are forms of degree $d$ in $n$ variables. Let $$\mathbf{x}=(\mathbf{y},\mathbf{z},\mathbf{w}),$$
 where $\mathbf{y}\in \N^m, \mathbf{z}\in \N^{s}, \mathbf{w}\in \N^t$ and
 $$m+s+t=n.$$
  Then each $F_i$ can be uniquely represented as
\begin{align}\label{representFintofgh2}F_i(\mathbf{y},\mathbf{z},\mathbf{w})=f_i(\mathbf{y})+g_i(\mathbf{y},\mathbf{z})
+h_i(\mathbf{y},\mathbf{z},\mathbf{w}),\end{align}
where the degree of $g_i$ as a polynomial of $\mathbf{y}$ is smaller than $d$ and the degree of $h_i$ as a polynomial of $(\mathbf{y},\mathbf{z})$ is also smaller than $d$, that is,
\begin{align*}\deg_{\mathbf{y}}(g_i)<d \ \textrm{ and } \ \deg_{(\mathbf{y},\mathbf{z})}(h_i)<d.\end{align*}
Note that \eqref{representFintofgh2} is the same as \eqref{representFintofgh1}. Furthermore, each $h_i$ can be uniquely represented as
\begin{align}\label{representhintoGH}h_i(\mathbf{y},\mathbf{z},\mathbf{w})=G_i(\mathbf{y},\mathbf{z},\mathbf{w})
+H_i(\mathbf{w}),\end{align}
where the degree of $G_i$ as a polynomial of $\mathbf{w}$ is smaller than $d$ and $H_i$ is a form of $\mathbf{w}$.
We write
\begin{align*}\mathbf{f}=(f_1,\ldots,f_R),\ \mathbf{g}=(g_1,\ldots,g_R)\  \textrm{ and } \ \mathbf{H}=(H_1,\ldots,H_R).\end{align*}

In the next result, we consider the integration of $S_{\mathbf{F}}(\boldsymbol{\alpha})$ over minor arcs, where
 the generating function $S_{\mathbf{F}}(\boldsymbol{\alpha})$ is given in \eqref{definegeneratingS} and from now on we choose
\begin{align}\label{chooselambda} \lambda_{j}(x)=\begin{cases}\Lambda(x)\ &\textrm{ if }\ b_j'P<x\le b_j''P,
\\ 0\ &\textrm{ otherwise}.\end{cases}
\end{align}

\begin{lemma}\label{lemmaoverminor1}
Let $\mathbf{F}=(F_1,\ldots,F_R)$ be forms of degree $d$ in $n$ variables with the unique representation as in \eqref{representFintofgh2} and \eqref{representhintoGH}. Let $L<\frac{1}{4}P^{\frac{Rd}{R+1}}$.
Let $\mathfrak{m}(L)$ be given in \eqref{definemQ}.
Suppose that  (i)
\begin{align}\label{definekappa1}m-\dim V_{\mathbf{f}}^\ast \ge \kappa_1:=(2^{d+2}dR^2+1)\cdot2^{d-1}(d-1)R(R+1)+1,\end{align}
 (ii)
 \begin{align}\label{definekappa2}m+s-\dim V_{\mathbf{g}}^\ast \ge \kappa_2:=(2^{d+2}dR^2+3)\cdot2^{d-1}(d-1)R(R+1)+1\end{align}
  and (iii)
  \begin{align}\label{rankH=R}\rank(\mathbf{H})=R.\end{align}
Then there exists a positive number $\delta=\delta_{d,R}>0$ such that
\begin{align*}\int_{\mathfrak{m}(L)}S_\mathbf{F}(\boldsymbol{\alpha})d\boldsymbol{\alpha} \ll P^{n-Rd}(\log P)^{7n}L^{-\delta}.\end{align*}
\end{lemma}
\begin{proof}
We define
\begin{align*}\mathfrak{n}(Q)=\begin{cases}\mathfrak{M}(2Q)\setminus \mathfrak{M}(Q),\ &\textrm{ if }\ Q<\frac{1}{4}P^{\frac{Rd}{R+1}},
\\ \mathfrak{m}(Q),\ & \textrm{ if }\ Q=\frac{1}{4}P^{\frac{Rd}{R+1}}.\end{cases}\end{align*}
By the dyadic argument, we only need to establish
\begin{align*}\int_{\mathfrak{n}}S_F(\boldsymbol{\alpha})d\boldsymbol{\alpha} \ll P^{n-Rd}(\log P)^{7n}Q^{-\delta}\end{align*}
for $L\le Q\le \frac{1}{4}P^{\frac{Rd}{R+1}}$ and $\mathfrak{n}\subseteq \mathfrak{n}(Q)$.

Note that for $\mathfrak{n}\subseteq \mathfrak{n}(Q)$, one has $|\mathfrak{n}|\ll Q^{R+1}P^{-Rd}$. On applying Proposition \ref{prop} and Lemma \ref{lemmaboundE2}, we obtain
\begin{align*}\int_{\mathfrak{n}}S_\mathbf{F}(\boldsymbol{\alpha})d\boldsymbol{\alpha} \ll\, &P^{n-Rd}(\log P)^{2n}X^{\frac{1}{2}(R+1)-\frac{\kappa_1}{2^{d}(d-1)R}}Q^{\frac{1}{2}(R+1)}
\\ &\ \ +P^{n-Rd}(\log P)^{2n}X^{\frac{3}{4}(R+1)-\frac{\kappa_2}{2^{d+1}(d-1)R}}Q^{\frac{1}{4}(R+1)}
\\ &\ \ +
P^{n-Rd}(\log P)^{7n}X^{R+1}Q^{-\frac{R+1}{2^{d+2}dR^2}+\varepsilon}.\end{align*}
We choose
$$X=Q^{\frac{1}{2^{d+2}dR^2+\eta}},$$
where $\eta=\eta_{d,R}>0$ is a sufficiently small positive real number depending on $d$ and $R$. We write
$$\Delta_0=\frac{1}{2^{d+2}dR^2+\eta}\cdot \frac{1}{2^{d}(d-1)R}.$$
Then we obtain
\begin{align*}\int_{\mathfrak{n}}S_\mathbf{F}(\boldsymbol{\alpha})d\boldsymbol{\alpha} \ll\, &P^{n-Rd}(\log P)^{7n}\big(Q^{-\Delta_0\cdot \Delta_1} +Q^{-\frac{1}{2}\Delta_0\cdot \Delta_2}+Q^{-\Delta_3+\varepsilon}\big),\end{align*}
where
\begin{align*}\Delta_1=\kappa_1-(2^{d+2}dR^2+\eta+1)\cdot 2^{d-1}(d-1)R(R+1),\end{align*}
\begin{align*}\Delta_2=\kappa_2-(2^{d+2}dR^2+\eta+3)\cdot2^{d-1}(d-1)R(R+1)\end{align*}
and
\begin{align*}\Delta_3=\frac{R+1}{2^{d+2}dR^2}-\frac{R+1}{2^{d+2}dR^2+\eta}.\end{align*}

Given the values of $\kappa_1$ in \eqref{definekappa1} and $\kappa_2$ in \eqref{definekappa2},  all of $\Delta_1,\Delta_2$ and $\Delta_3$ are positive
 provided that
\begin{align*}\eta<\frac{1}{2^{d-1}(d-1)R(R+1)}.\end{align*} Therefore, there exists $\delta_{d,R}>0$ such that
\begin{align*}\int_{\mathfrak{n}}S_F(\boldsymbol{\alpha})d\boldsymbol{\alpha}  \ll\,
P^{n-Rd}(\log P)^{7n}Q^{-\delta_{d,R}}.\end{align*}
This completes the proof.
\end{proof}

Lemma \ref{lemmaoverminor1} provides an acceptable upper bound for the mean value of $S_F(\boldsymbol{\alpha})$ over minor arcs by assuming some (unsatisfactory) technical conditions in \eqref{representFintofgh2}, \eqref{representhintoGH}, \eqref{definekappa1}, \eqref{definekappa2} and \eqref{rankH=R}. The purpose of the next section is to remove these technical conditions.
 \vskip3mm

\section{Geometric considerations}

We introduce the codimension of the singular locus (in the sense of Birch)
\begin{align}\label{definecodim}\codim V_{\mathbf{F}}^\ast=n-\dim V_{\mathbf{F}}^\ast,\end{align}
where $\mathbf{F}=\mathbf{F}(x_1,\ldots,x_n)$ is a system of forms in $n$ variables.
\begin{lemma}\label{lemmaincrease}Let $F_1,\ldots,F_R$ be forms in $\mathbf{x}=(x_1,\ldots,x_n)$, and let $\mathbf{F}=(F_1,\ldots,F_R)$. Let
$\mathbf{f}=\mathbf{f}(\mathbf{y})$ be a system of forms in $\mathbf{y}=(y_1,\ldots,y_{n-1})$ defined to be
$$\mathbf{f}(y_1,\ldots,y_{n-1})=\mathbf{F}(y_1,\ldots,y_{n-1},0).$$
Then we have
$$\codim V_{\mathbf{f}}^\ast\ge \codim V_{\mathbf{F}}^\ast-R-1.$$\end{lemma}
\begin{proof}This is Lemma 3 in \cite{CM}. Cook and Magyar provided the inequality $\codim V_{\mathbf{f}}^\ast\ge \codim V_{\mathbf{F}}^\ast-R$ due to an oversight. Their proof in fact implies $\codim V_{\mathbf{f}}^\ast\ge \codim V_{\mathbf{F}}^\ast-R-1$.
\end{proof}

Let
\begin{align}\label{kappa1}\kappa_1=(2^{d+2}dR^2+1)\cdot2^{d-1}(d-1)R(R+1)+1\end{align}
and
 \begin{align}\label{kappa2} \kappa_2=(2^{d+2}dR^2+3)\cdot2^{d-1}(d-1)R(R+1)+1.\end{align}
They coincide with those in Lemma \ref{lemmaoverminor1}. Let
 \begin{align}\label{kappa31} \kappa_3=(R+1)\kappa_2+\kappa_1+dR(R+1)^2+2R(R+1)+R.\end{align}
We point out that (for $d\ge 2$ and $R\ge 1$)
 \begin{align}\label{kappa33} 4^{d+2}d^2R^5\ge \kappa_3.\end{align}

\begin{lemma}\label{lemmakey}Let $\mathbf{F}=(F_1,\ldots,F_R)$ be forms of degree $d$ in $n$ variables. Let $\kappa_1$ and $\kappa_2$ be given in \eqref{kappa1} and \eqref{kappa2}. Suppose that $F_1,\ldots,F_R$ is a non-singular system and that
\begin{align*}n\ge 4^{d+2}d^2R^5.\end{align*}
Then up to the permutation of variables, $\mathbf{F}$ can be represented in the form \eqref{representFintofgh2} and \eqref{representhintoGH} such that (i) $m-\dim V_{\mathbf{f}}^\ast \ge \kappa_1$,  (ii) $m+s-\dim V_{\mathbf{g}}^\ast \ge \kappa_2$
  and (iii) $\rank(\mathbf{H})=R$.
\end{lemma}
\begin{proof}In fact, we shall prove the desired conclusion subject to the weaker condition
\begin{align}\label{conditioncodimF}n\ge \kappa_3,\end{align}
and this completes the proof the lemma in view of \eqref{kappa33}. We use $4^{d+2}d^2R^5$ instead of $\kappa_3$ in order to make the statement clean.

We have $\rank(\mathbf{F})=R$, and we can choose $R$ coordinates of $\mathbf{F}$ such that the corresponding coefficients matrix $(F_{\mathbf{j}_k}^{(i)})_{1\le i,k\le R}$ is invertible. Note that the coordinate of $\mathbf{F}$ is as in \eqref{mathbfj}, and each coordinate involves at most $d$ variables.
Let
$$t=dR.$$
 Therefore, up to the permutation of variables, we may assume that
 $$\rank(\mathbf{H})=R,$$
 where
 $$\mathbf{H}(x_{n-t+1},\ldots,x_{n})=\mathbf{F}(0,\ldots,0,x_{n-t+1},\ldots,x_{n})$$
 is a system of forms in $t$ variables.

Let
\begin{align}\label{definem}m=n-2R-dR(R+1)-\kappa_2\ \ \textrm{ and }\ \ s=n-m-t.\end{align}
We define
\begin{align*}Y=\{\mathbf{x}\in \A^n:\ \rank\big(\frac{\partial F_i}{\partial x_j}(\mathbf{x})\big)_{\substack{1\le i\le R \\ m+1\le j\le n}}<R\}.\end{align*}

We introduce $n$-dimensional vectors
$$\gamma_i(\mathbf{x})=\nabla F_i=(\frac{\partial F_i}{\partial x_1}(\mathbf{x}), \ldots, \frac{\partial F_i}{\partial x_n}(\mathbf{x}))$$
and $(n-m)$-dimensional vectors
$$\eta_i(\mathbf{x})=(\frac{\partial F_i}{\partial x_{m+1}}(\mathbf{x}), \ldots, \frac{\partial F_i}{\partial x_n}(\mathbf{x})).$$
Then for any $\boldsymbol{\beta}=(\beta_1,\cdots,\beta_R)\in \P_{\C}^{R-1}$, we introduce
$$V_{\boldsymbol{\beta}}=\{\mathbf{x}\in \A^n:\ \sum_{i=1}^R\beta_i\gamma_i(\mathbf{x})=\mathbf{0}\}$$
and
$$Y_{\boldsymbol{\beta}}=\{\mathbf{x}\in \A^n:\ \sum_{i=1}^R\beta_i\eta_i(\mathbf{x})=\mathbf{0}\}.$$
Here $\P_{\C}^{r}$ denotes the projective space. By Proposition 7.1 in Chapter I \cite{Hbook}, one has
\begin{align}\label{dimVYbeta}\dim Y_{\boldsymbol{\beta}} \le \dim V_{\boldsymbol{\beta}}+m.\end{align}

Since $\mathbf{F}$ is a non-singular system, we have
\begin{align}\label{dimVFast}\dim V_{\mathbf{F}}^\ast \le R,\end{align}
and by \eqref{definecodim},
\begin{align}\label{codimVF}\codim V_{\mathbf{F}}^\ast\ge n-R.\end{align}
Note that $V_{\boldsymbol{\beta}} \subseteq V_{\mathbf{F}}^\ast$. By \eqref{dimVYbeta} and \eqref{dimVFast}, we have
\begin{align}\label{dimYbeta}\dim Y_{\boldsymbol{\beta}} \le m+R \ \textrm{ for any }\ \boldsymbol{\beta}\in \P_{\C}^{R-1}.\end{align}

We claim that
\begin{align}\label{dimYandYbeta}\dim Y\le R+\max_{\boldsymbol{\beta}\in \P_{\C}^{R-1}}\dim Y_{\boldsymbol{\beta}}.\end{align}
Let $U$ be an (arbitrary) irreducible component of $Y$, and let $U_{\boldsymbol{\beta}}=U\cap Y_{\boldsymbol{\beta}}$. The claim \eqref{dimYandYbeta} follows from
\begin{align}\label{dimUandYbeta}\dim U\le R+\max_{\boldsymbol{\beta}\in \P_{\C}^{R-1}}\dim Y_{\boldsymbol{\beta}}.\end{align}

Let
$$r=\max\{\rank\big(\frac{\partial F_i}{\partial x_j}(\mathbf{x})\big)_{\substack{1\le i\le R \\ m+1\le j\le n}}:\ \mathbf{x}\in U\}.$$
By the definition of $Y$, one has $r<R$. If $r=0$, then $U=U_{\boldsymbol{\beta}}$ for all $\boldsymbol{\beta}\in \P_{\C}^{R-1}$.  In particular, \eqref{dimUandYbeta} follows immediately. Next we assume $r\ge 1$. Without loss of generality, we may assume that
$$\rank\big(\frac{\partial F_i}{\partial x_j}(\mathbf{x})\big)_{\substack{1\le i\le r \\ m+1\le j\le m+r}}=r$$
for some (nonzero) $\mathbf{x}\in U$. Then we introduce
$$U_0=\big\{\mathbf{x}\in U:\ \rank\big(\frac{\partial F_i}{\partial x_j}(\mathbf{x})\big)_{\substack{1\le i\le r \\ m+1\le j\le m+r}}=r\big\}.$$

For each $\mathbf{x}\in U_0$, there exists a unique $\boldsymbol{\alpha}\in \P_{\C}^{r}$ such that $\sum_{i=1}^{r+1}\alpha_i\eta_i(\mathbf{x})=\mathbf{0}$. This induces a morphism $\rho: U_0\rightarrow \P_{\C}^{r}$. In fact, $\rho$ is in the form $\rho=(u_1,\ldots,u_{r+1})$, where $u_1,\ldots,u_{r+1}$ are homogeneous polynomials of the same degree. Therefore, for any $\boldsymbol{\alpha}\in \rho(U_0)$, the fiber
$\rho^{-1}(\boldsymbol{\alpha})$ can be expressed as
$$\rho^{-1}(\boldsymbol{\alpha})=\{\mathbf{x}\in U_0:\ u_1(\mathbf{x})=\alpha_1,\ldots,u_{r+1}(\mathbf{x})=\alpha_{r+1}\}.$$
It follows from Proposition 7.1 in Chapter I \cite{Hbook} that
\begin{align}\label{dimY0alpha}\dim U_0 \le r+1+\dim \rho^{-1}(\boldsymbol{\alpha}).\end{align}
If we write $\boldsymbol{\beta}=(\alpha_1,\ldots,\alpha_{r+1},0,\ldots,0)\in \P_{\C}^{R-1}$, then
\begin{align}\label{alphabeta}\rho^{-1}(\boldsymbol{\alpha}) \subseteq Y_{\boldsymbol{\beta}}.\end{align}
We conclude  from \eqref{dimY0alpha} and \eqref{alphabeta} that
\begin{align}\label{dimU0}\dim U_0\le R+\max_{\boldsymbol{\beta}\in \P_{\C}^{R-1}}\dim Y_{\boldsymbol{\beta}}.\end{align}
Note that $U_0\not=\emptyset$ is open in $U$. By Proposition 1.10 in Chapter I \cite{Hbook}(see also Example 1.1.3 in Chapter I \cite{Hbook}), we have $\dim U=\dim U_0$, and \eqref{dimUandYbeta} follows from \eqref{dimU0}. Now the claim \eqref{dimYandYbeta} has been established.

It follows from \eqref{dimYbeta} and \eqref{dimYandYbeta} that
\begin{align}\label{dimY}\dim Y\le m+2R.\end{align}
We define the system of polynomials $\mathbf{f}(x_1,\ldots,x_m)$ in $m$ variables by
$$\mathbf{f}(x_1,\ldots,x_m)=
\mathbf{F}(x_1,\ldots,x_m,0,\ldots,0),$$
and define the system of polynomials $\mathbf{L}(x_1,\ldots,x_n)$ in $n$ variables by
$$\mathbf{L}(x_1,\ldots,x_n)=\mathbf{F}(x_1,\ldots,x_n)-\mathbf{F}(x_1,\ldots,x_m,0,\ldots,0).$$

We deduce from Lemma \ref{lemmaincrease} and \eqref{codimVF} that
\begin{align*}\codim V_{\mathbf{f}}^\ast\ge n-R-(n-m)(R+1),\end{align*}
and by \eqref{definem},
\begin{align*}\codim V_{\mathbf{f}}^\ast\ge n-R-\big(2R+dR(R+1)+\kappa_2\big)(R+1).\end{align*}
In particular, by \eqref{kappa31} and \eqref{conditioncodimF}, one has
\begin{align*}\codim V_{\mathbf{f}}^\ast\ge \kappa_1.\end{align*}

Note that $V_{\mathbf{L}}^\ast\subseteq Y$. By \eqref{definecodim} and \eqref{dimY}, one has
\begin{align*}\codim V_{\mathbf{L}}^\ast\ge n-(m+2R),\end{align*}
and by \eqref{definem},
\begin{align}\label{codimVL}\codim V_{\mathbf{L}}^\ast\ge \kappa_2+dR(R+1).\end{align}

We write $\mathbf{v}=(x_{1},\ldots,x_{n-t})$ and $\mathbf{w}=(x_{n-t+1},\ldots,x_{n})$. Then each $L_i$ can be uniquely represented as
\begin{align*}L_i(\mathbf{v},\mathbf{w})=g_i(\mathbf{v})
+G_i(\mathbf{v},\mathbf{w})+H_i(\mathbf{w}),\end{align*}
where $g_i$ is a form of $\mathbf{v}$, $H_i$ is a form of $\mathbf{w}$, the degree of $G_i$ as a polynomial of $\mathbf{v}$ and the degree of $G_i$ as a polynomial of $\mathbf{w}$ are both smaller than $d$, that is,
\begin{align*}\deg_{\mathbf{v}}(G_i)<d \ \textrm{ and } \ \deg_{\mathbf{w}}(G_i)<d.\end{align*}

We conclude from Lemma \ref{lemmaincrease} and \eqref{codimVL} that
\begin{align*}\codim (V_{\mathbf{g}}^\ast)\ge \kappa_2.\end{align*}
The proof of the lemma is complete.
\end{proof}

Now Lemma \ref{lemmaoverminor1} and Lemma \ref{lemmakey} together imply the following.
\begin{lemma}\label{lemmaoverminor2}
Let $\mathbf{F}=(F_1,\ldots,F_R)$ a system of forms of degree $d$ in $n$ variables. Let $L<\frac{1}{4}P^{\frac{Rd}{R+1}}$.
Let $\mathfrak{m}(L)$ be defined in \eqref{definemQ}. Suppose that $F_1,\ldots,F_R$ is a non-singular system and that
\begin{align*}n\ge 4^{d+2}d^2R^5.\end{align*}
Then there exists a positive number $\delta=\delta_{d,R}>0$ such that
\begin{align*}\int_{\mathfrak{m}(L)}S_\mathbf{F}(\boldsymbol{\alpha})d\boldsymbol{\alpha} \ll P^{n-Rd}(\log P)^{7n}L^{-\delta}.\end{align*}
\end{lemma}
 \vskip3mm

\section{The singular series and the singular integral}

We define the multiple Gauss sum
\begin{align}\label{definegausssum}S^\ast(q,\mathbf{a}):=S_\mathbf{F}^\ast(q,\mathbf{a})=
\sum_{\substack{1\le \mathbf{b}\le q\\ (\mathbf{b},q)=1}}e\big(
\frac{\mathbf{a}\cdot \mathbf{F}(\mathbf{b})}{q}\big),\end{align}
and then we define
\begin{align}\label{defineAq}\mathcal{A}^\ast(q)=\sum_{\substack{1\le \mathbf{a}\le q\\ (a_1,\ldots,a_R,q)=1}}S^\ast(q,\mathbf{a}).\end{align}
\begin{lemma}\label{lemmamultiplicative}Suppose that $(q_1,q_2)=1$. Then we have
\begin{align*}\mathcal{A}^\ast(q_1q_2)=\mathcal{A}^\ast(q_1)\mathcal{A}^\ast(q_2).\end{align*}
\end{lemma}
\begin{proof}It can be proved by the standard argument.\end{proof}

In view of Lemma \ref{lemmamultiplicative}, in order to obtain the upper bound for $\mathcal{A}^\ast(q)$, we only need to study $S^\ast(p^r,\mathbf{a})$, where $p$ is a prime and $r\in \N$.

We use the notation
$$[1,n]:=\{1\le i\le n:\ i\in \N\}.$$
For a nonempty index set $I=\{i_1,\ldots,i_m\}\subseteq [1,n]$, we introduce the vector
\begin{align*}\mathbf{x}_I=(x_{i_1},\ldots,x_{i_m}).\end{align*}
\begin{lemma}\label{lemmaVgauss}Suppose that $p$ is a prime, $(a_1,,\ldots,a_R,p)=1$ and $r\ge 1$. Let $q=p^r$. If
\begin{align}\label{Gausssumcondition}n-\dim V^\ast_{\mathbf{F}} \ge  2^{d}d(R+1)(R+2),\end{align}
then
\begin{align*}S^\ast(q,\mathbf{a})\ll q^{n-2-R}.\end{align*}
\end{lemma}
\begin{proof}
For an index set $I\subseteq [1,n]$, we write $I^c=[1,n]\setminus I$ and define
\begin{align*}S(I):=S(I;p^r,\mathbf{a})=\sum_{\substack{1\le \mathbf{x}_I\le p^{r-1}}}
\sum_{\substack{1\le \mathbf{x}_{I^c}\le p^r}}e\big(
\frac{\mathbf{a}\cdot \mathbf{F}(\mathbf{w})}{p^r}\big),\end{align*}
where the first summation is taken over $x_i$ with $i\in I$,
the second summation is taken over $x_j$ with  $j\in I^c$ and the vector $\mathbf{w}$ means
$\mathbf{w}=(w_1,\ldots,w_n)$ with $w_i=px_i$ for $i\in I$ and $w_j=x_j$ for $j\in I^c$.

We observe
\begin{align*}S^\ast(p^r,\mathbf{a})=\sum_{I}(-1)^{|I|}S(I).\end{align*}
In particular, we have
\begin{align}\label{appearSI}|S^\ast(p^r,\mathbf{a})|\le 2^{n}\sup_{I}|S(I)|.\end{align}

Therefore, we need to establish the upper bound for $S(I)$. Without loss of generality, we assume that
$$I=\{i\in \N:\ 1\le i\le m\},$$
and we write
$$l=n-m.$$
Let
$$\mathbf{y}=(y_1,\ldots,y_{m})\in \Z^{m}\  \textrm{ and } \ \mathbf{z}=(z_1,\ldots,z_{l})\in \Z^{l}.$$
Then we have
\begin{align*}S(I)=\sum_{\substack{1\le \mathbf{y}\le p^{r-1}}}
\sum_{\substack{1\le \mathbf{z}\le p^r}}e\big(
\frac{\mathbf{a}\cdot \mathbf{F}(p\mathbf{y},\mathbf{z})}{p^r}\big).\end{align*}

We first deal with the case $r\ge 2d$. We have
\begin{align*}S(I)=\sum_{\substack{1\le \mathbf{b}\le p}}\sum_{\substack{1\le \mathbf{y}\le p^{r-1}}}
\sum_{\substack{1\le \mathbf{z}\le p^r \\ \mathbf{z}\equiv \mathbf{b}\pmod{p}}}e\big(
\frac{\mathbf{a}\cdot \mathbf{F}(p\mathbf{y},\mathbf{z})}{p^r}\big).\end{align*}
Then we deduce by changing variables that
\begin{align}\label{boundSItoSIb}S(I)=\sum_{\substack{1\le \mathbf{b}\le p}}S_{\mathbf{b}}(I),\end{align}
where
\begin{align*}S_{\mathbf{b}}(I)=\sum_{\substack{1\le \mathbf{y}\le p^{r-1}}}
\sum_{\substack{1\le \mathbf{z}\le p^{r-1} }}e\big(
\frac{\mathbf{a}\cdot \mathbf{F}(p\mathbf{y},p\mathbf{z}+\mathbf{b})}{p^r}\big).\end{align*}
Note that $S_{\mathbf{b}}(I)$ can be represented in the form
\begin{align*}S_{\mathbf{b}}(I)=\sum_{\substack{1\le \mathbf{y}\le p^{r-1}}}
\sum_{\substack{1\le \mathbf{z}\le p^{r-1} }}e\big(
\frac{\mathbf{a}\cdot \mathbf{F}(\mathbf{y},\mathbf{z})}{p^{r-d}}+G_{\mathbf{b},p,\mathbf{a}}(\mathbf{y},\mathbf{z})\big),\end{align*}
where as a polynomial of $(\mathbf{y},\mathbf{z})$, the degree of $G_{\mathbf{b},p,\mathbf{a}}=G_{\mathbf{b},p,\mathbf{a}}(\mathbf{y},\mathbf{z})$ is smaller than $d$. Let
\begin{align*}\kappa= 2^{d}d(R+1)(R+2),\end{align*}
Then by Lemma \ref{lemmaboundS}, we have
\begin{align}\label{boundSI1}S_{\mathbf{b}}(I)\ll p^{n(r-1)}q^{\varepsilon}p^{-(r-d)\frac{\kappa}{2^{d-1}(d-1)R}}.\end{align}
In the case $r\ge 2d$, we have $r-d\ge \frac{r}{2}$ and by \eqref{boundSI1},
\begin{align}\label{boundSI2}S_{\mathbf{b}}(I)\ll p^{n(r-1)}q^{\varepsilon-\frac{\kappa}{2^{d}(d-1)R}}.\end{align}
Inserting \eqref{boundSI2} into \eqref{boundSItoSIb}, we obtain
\begin{align*}S(I)\ll q^{n+\varepsilon-\frac{\kappa}{2^{d}R(d-1)}}\ll q^{n-R-2}.\end{align*}

We next deal with the case $r<2d$. We have the trivial bound
\begin{align*}S(I)\ll p^{nr-m}\ll p^{nr-\frac{1}{2d}mr}\ll q^{n-\frac{1}{2d}m}.\end{align*}
Thus if $m\ge 2d(R+2)$, then $S(I)\ll q^{n-R-2}$. Now we consider the case $m<2d(R+2)$.
We represent $\mathbf{F}(p\mathbf{y},\mathbf{z})$ as
$$\mathbf{F}(p\mathbf{y},\mathbf{z})=\mathbf{f}(\mathbf{z})+\mathbf{g}(\mathbf{z},\mathbf{y}),$$
where the degree of $g_i$ as a polynomial of $\mathbf{z}$ is smaller than $d$.

We deduce by Lemma \ref{lemmaincrease} that
$$\codim V_{\mathbf{f}}^\ast\ge \kappa-2d(R+2)(R+1)\ge 2^{d-1}dR(R+2).$$
 Then by Lemma \ref{lemmaboundS}, we have
\begin{align*}
\sum_{\substack{1\le \mathbf{z}\le p^r}}e\big(
\frac{\mathbf{a}\cdot \mathbf{F}(p\mathbf{y},\mathbf{z})}{p^r}\big)\ll q^{l+\varepsilon-\frac{d(R+2)}{d-1}},\end{align*}
and therefore,
\begin{align*}S(I)\ll q^{n-R-2}.\end{align*}

We have established the estimate $S(I)\ll q^{n-R-2}$ (in all cases) and by \eqref{appearSI},
 \begin{align*}S^\ast(q,\mathbf{a})\ll q^{n-2-R}.\end{align*}
This completes the proof.
\end{proof}

\begin{lemma}\label{lemmaboundAq}Let $q\ge 1$. If \eqref{Gausssumcondition} holds,
then one has
\begin{align}\label{boundAq}\mathcal{A}^\ast(q)\ll q^{n-2+\varepsilon}.\end{align}
\end{lemma}
\begin{proof}If $q$ is a power of a prime, then the inequality \eqref{boundAq} follows from Lemma \ref{lemmaVgauss}. In view of Lemma \ref{lemmamultiplicative}, we obtain \eqref{boundAq} for all $q$.\end{proof}

Now we are able to define the singular series subject to the condition \eqref{Gausssumcondition}. Let
\begin{align}\label{definesingularseries}\mathfrak{S}_{\mathbf{F}}^\ast=\sum_{q=1}^\infty \frac{1}{\phi^n(q)}
\mathcal{A}^\ast(q).\end{align}

 The singular integral in this paper is as the same as that in Birch's work \cite{Birch}.
 Thus, we introduce the singular integral briefly.

Define
\begin{align*}\upsilon(\boldsymbol{\beta})=\int_{\mathfrak{B}}e\big(\boldsymbol{\beta}\cdot \mathbf{F}(\mathbf{x})\big)d\mathbf{x}.\end{align*}
Subject to \eqref{Gausssumcondition}, we have by Lemma 5.2 of Birch \cite{Birch} that
\begin{align*}\upsilon(\boldsymbol{\beta})\ll (1+|\boldsymbol{\beta}|_{\max})^{-R-2},\end{align*}
where $|\boldsymbol{\beta}|_{\max}=\max\{|\beta_1|,\ldots,|\beta_R|\}$.

 Now we define the singular integral
\begin{align}\label{definesingularintegral}\mathfrak{I}_{\mathbf{F}}=\lim_{Q\rightarrow +\infty}\int_{|\boldsymbol{\beta}|\le Q}v(\boldsymbol{\beta})d\boldsymbol{\beta}.\end{align}
The following result is Lemma 5.3 of Birch \cite{Birch}.
\begin{lemma}\label{lemmaapproximateI}Suppose that \eqref{Gausssumcondition} holds. Then we have
\begin{align*}\Big|\mathfrak{I}_{\mathbf{F}}-\int_{|\boldsymbol{\beta}|\le Q}v(\boldsymbol{\beta})d\boldsymbol{\beta}\Big|\ll Q^{-\frac{1}{2}}.\end{align*}
\end{lemma}
\vskip3mm

\section{Proof of Theorem \ref{theorem2}}

On recalling the generating function $S_{\mathbf{F}}(\boldsymbol{\alpha})$ introduced in \eqref{definegeneratingS} and \eqref{chooselambda}, we conclude that
\begin{align}\label{circle}N_{\mathbf{F}}(P)=\int_{0}^1S_{\mathbf{F}}(\boldsymbol{\alpha})d\boldsymbol{\alpha},\end{align}

We define \begin{align}\label{definecalL}\mathcal{L}=(\log P)^{A_0}\end{align} with $A_0>0$ sufficiently large. Then we introduce the major arcs and minor arcs
\begin{align*}\mathfrak{M}=\mathfrak{M}(\mathcal{L}) \ \textrm{ and } \mathfrak{m}=\mathfrak{m}(\mathcal{L}),\end{align*}
where $\mathfrak{M}(\mathcal{L})$ and $\mathfrak{m}(\mathcal{L})$ are given in \eqref{defineMQ} and \eqref{definemQ}, respectively.

With above notations, we conclude that
\begin{align}\label{circleMm}N_{\mathbf{F}}(P)=\int_{\mathfrak{M}}S_{\mathbf{F}}(\boldsymbol{\alpha})d\boldsymbol{\alpha}+
\int_{\mathfrak{m}}S_{\mathbf{F}}(\boldsymbol{\alpha})d\boldsymbol{\alpha}.\end{align}

We first deal with the contribution from the major arcs.
\begin{lemma}\label{lemmamajorcontribution}Suppose that \eqref{Gausssumcondition} holds. Then we have
\begin{align}\label{asympmajor}\int_{\mathfrak{M}}S_{\mathbf{F}}(\boldsymbol{\alpha})d\boldsymbol{\alpha}=
\mathfrak{S}_{\mathbf{F}}^\ast \mathfrak{I}_{\mathbf{F}}P^{n-Rd}+O(P^{n-Rd}\mathcal{L}^{-\frac{1}{2}}).\end{align}\end{lemma}
\begin{proof}By the definition of the major arcs $\mathfrak{M}$,  we have
\begin{align}\label{inttosum}\int_{\mathfrak{M}}S_{\mathbf{F}}(\boldsymbol{\alpha})d\boldsymbol{\alpha}=
\sum_{q\le \mathcal{L}}\sum_{\substack{1\le \mathbf{a}\le q\\ (a_1,\ldots,a_R,q)=1}}\int_{|\boldsymbol{\beta}|\le \frac{\mathcal{L}}{qP^d}}
S_{\mathbf{F}}(\frac{\mathbf{a}}{q}+\boldsymbol{\beta})d\boldsymbol{\beta}.\end{align}
By the standard application of the Siegel-Walfisz theorem and the partial summation formula, for $q\le \mathcal{L}$ and $|\boldsymbol{\beta}|\le \mathcal{L}(qP^d)^{-1}$, we can deduce that
\begin{align}\label{Saqbeta}S_{\mathbf{F}}(\frac{\mathbf{a}}{q}+\boldsymbol{\beta})=\frac{1}{\phi(q)^n}S^\ast(q,\mathbf{a})
I_P(\boldsymbol{\beta})+O(P^{n}\mathcal{L}^{-101-R}),\end{align}
where the Gauss sum $S^\ast(q,\mathbf{a})$ is given in \eqref{definegausssum} and
\begin{align*}I_P(\boldsymbol{\beta})=\int_{P\cdot \mathfrak{B}}e\big(\boldsymbol{\beta}\cdot\mathbf{F}(\mathbf{x})\big)d\mathbf{x}.\end{align*}

We put \eqref{Saqbeta} into \eqref{inttosum} to obtain
\begin{align*}\int_{\mathfrak{M}}S_{\mathbf{F}}(\boldsymbol{\alpha})d\boldsymbol{\alpha}=&
\sum_{q\le \mathcal{L}}\frac{1}{\phi(q)^n}\sum_{\substack{1\le \mathbf{a}\le q\\ (a_1,\ldots,a_R,q)=1}}S^\ast(q,a)\int_{|\boldsymbol{\beta}|\le \frac{\mathcal{L}}{qP^d}}I_P(\boldsymbol{\beta})d\boldsymbol{\beta}
\\ & \ \ +O(P^{n-Rd}\mathcal{L}^{-100}).\end{align*}
Invoking the notation \eqref{defineAq}, we get
\begin{align}\label{invokingAq}\int_{\mathfrak{M}}S_{\mathbf{F}}(\boldsymbol{\alpha})d\boldsymbol{\alpha}=
\sum_{q\le \mathcal{L}}\frac{1}{\phi(q)^n}\mathcal{A}^\ast(q)\int_{|\boldsymbol{\beta}|\le \frac{\mathcal{L}}{qP^d}}I_P(\boldsymbol{\beta})d\boldsymbol{\beta}+O(P^{n-Rd}\mathcal{L}^{-100}).\end{align}

By changing variables, we have
\begin{align*}I_{P}(\boldsymbol{\beta})=P^{n}\upsilon(P^d\boldsymbol{\beta}).\end{align*}
Another change of variables implies
\begin{align}\label{changingvariables}\int_{|\boldsymbol{\beta}|\le \frac{\mathcal{L}}{qP^d}}I_P(\boldsymbol{\beta})d\boldsymbol{\beta} =P^{n-Rd}\int_{|\boldsymbol{\beta}|\le \frac{\mathcal{L}}{q}}\upsilon(\boldsymbol{\beta})d\boldsymbol{\beta}.\end{align}
We deduce by \eqref{invokingAq} and \eqref{changingvariables} that
\begin{align}\label{afterchange}\int_{\mathfrak{M}}S_{\mathbf{F}}(\boldsymbol{\alpha})d\boldsymbol{\alpha}=
P^{n-Rd}\sum_{q\le \mathcal{L}}\frac{1}{\phi(q)^n}\mathcal{A}^\ast(q)\int_{|\boldsymbol{\beta}|\le \frac{\mathcal{L}}{q}}\upsilon(\boldsymbol{\beta})d\boldsymbol{\beta}+O(P^{n-Rd}\mathcal{L}^{-100}).\end{align}

Now we make use of Lemma  \ref{lemmaapproximateI} to conclude
\begin{align}\label{inssingularint}\int_{|\boldsymbol{\beta}|\le \frac{\mathcal{L}}{q}}\upsilon(\boldsymbol{\beta})d\boldsymbol{\beta}=
\mathfrak{I}_{\mathbf{F}}+O((q\mathcal{L}^{-1})^{\frac{1}{2}}).\end{align}
Then we deduce from \eqref{afterchange} and \eqref{inssingularint} that  \begin{align*}\int_{\mathfrak{M}}S_{\mathbf{F}}(\boldsymbol{\alpha})d\boldsymbol{\alpha}=
P^{n-Rd}\sum_{q\le \mathcal{L}}\frac{1}{\phi(q)^n}\mathcal{A}^\ast(q) \mathfrak{I}_{\mathbf{F}}+P^{n-Rd}\sum_{q\le \mathcal{L}}\frac{O((q\mathcal{L}^{-1})^{\frac{1}{2}})}{\phi(q)^n}|\mathcal{A}^\ast(q)|.\end{align*}
We apply Lemma \ref{lemmaboundAq} to conclude
\begin{align*}\sum_{q\le \mathcal{L}}\frac{q^{\frac{1}{2}}}{\phi(q)^n}|\mathcal{A}^\ast(q)|\ll 1,\end{align*}
 and therefore,
\begin{align}\label{applygauss1}\int_{\mathfrak{M}}S_{\mathbf{F}}(\boldsymbol{\alpha})d\boldsymbol{\alpha}=
P^{n-Rd}\sum_{q\le \mathcal{L}}\frac{1}{\phi(q)^n}\mathcal{A}^\ast(q) \mathfrak{I}_{\mathbf{F}}+O(P^{n-Rd}\mathcal{L}^{-\frac{1}{2}}).\end{align}
Applying Lemma \ref{lemmaboundAq} again, we obtain
\begin{align*}\sum_{q\le \mathcal{L}}\frac{1}{\phi(q)^n}\mathcal{A}^\ast(q)=
\sum_{q=1}^\infty\frac{1}{\phi(q)^n}\mathcal{A}^\ast(q)+O(\mathcal{L}^{-1+\varepsilon}).\end{align*}
On recalling the definition of the singular integral $\mathfrak{S}_{\mathbf{F}}^\ast$ in \eqref{definesingularintegral}, we obtain
\begin{align}\label{insertingsingularseries}\sum_{q\le \mathcal{L}}\frac{1}{\phi(q)^n}\mathcal{A}^\ast(q)=\mathfrak{S}_{\mathbf{F}}^\ast
+O(\mathcal{L}^{-1+\varepsilon}).\end{align}
We complete the proof by noting that \eqref{applygauss1} and \eqref{insertingsingularseries} together yield \eqref{asympmajor}.
\end{proof}

\noindent {\it Proof of Theorem \ref{theorem2}}. We deduce from Lemma \ref{lemmaoverminor2} that
\begin{align}\label{minor}\int_{\mathfrak{m}(\mathcal{L})}S_\mathbf{F}(\boldsymbol{\alpha})d\boldsymbol{\alpha} \ll P^{n-Rd}(\log P)^{7n}\mathcal{L}^{-\delta}\end{align}
for some small $\delta=\delta_{d,R}>0$.  Then we conclude from \eqref{lemmamajorcontribution}, \eqref{circleMm} and \eqref{minor} that
\begin{align*}N_{\mathbf{F}}(P)=
\mathfrak{S}_{\mathbf{F}}^\ast \mathfrak{I}_{\mathbf{F}}P^{n-Rd}+O\big(P^{n-Rd}(\log P)^{7n}\mathcal{L}^{-\delta}\big).\end{align*}
We complete the proof of Theorem \ref{theorem2} by choosing $A_0$ sufficiently large in \eqref{definecalL}.


\vskip4mm
\end{document}